\documentclass[final,3p,times,12pt]{elsarticle}
\journal{xxx}

\usepackage{multirow}
\usepackage{color}
\usepackage{chngcntr}
\usepackage{amsthm,amssymb,mathtools}

\theoremstyle{plain}
\newtheorem{theorem}{Theorem}
\newtheorem{corollary}{Corollary}
\newtheorem{proposition}{Proposition}
\newtheorem{lemma}{Lemma}
\theoremstyle{definition}
\newtheorem{definition}{Definition}
\newtheorem{example}{Example}
\newtheorem{hyp}{Assumption}
\theoremstyle{remark}
\newtheorem{remark}{Remark}

\newcommand{\Bc}{\mathcal B}
\newcommand{\Cc}{\mathcal C}
\newcommand{\Dc}{\mathcal D}
\newcommand{\Ec}{\mathcal E}
\newcommand{\Fc}{\mathcal F}
\newcommand{\Gc}{\mathcal G}
\newcommand{\Hc}{\mathcal H}
\newcommand{\Kc}{\mathcal K}
\newcommand{\Lc}{\mathcal L}
\newcommand{\Mc}{\mathcal M}

\newcommand{\Oc}{\mathcal O}
\newcommand{\Pc}{\mathcal P}
\newcommand{\Uc}{\mathcal U}
\newcommand{\Vc}{\mathcal V}
\newcommand{\Xc}{\mathcal X}
\newcommand{\Yc}{\mathcal Y}
\newcommand{\Zc}{\mathcal Z}

\newcommand{\Bb}{\mathbb B}
\newcommand{\Eb}{\mathbb E}
\newcommand{\Nb}{\mathbb N}
\newcommand{\Pb}{\mathbb P}
\newcommand{\Qb}{\mathbb Q}
\newcommand{\Rb}{\mathbb R}

\newcommand{\veps}{\varepsilon}
\newcommand{\vphi}{\varphi}

\newcommand{\Abs}[1]{\left\lvert #1 \right\rvert}

\newcommand{\brak}[2]{\langle #1 , #2 \rangle}

\newcommand{\Enbrace}[1]{\left\{ #1 \right\}}

\newcommand{\lb}{\langle}
\newcommand{\rb}{\rangle}

\newcommand{\Id}{\mathrm{Id}}				

\newcommand{\indic}[1]{\mathbf{1}_{#1}}			

\newcommand{\prob}[1]{\mathbb{P}\big{\{}#1\big{\}}} 	
\newcommand{\esp}[1]{\mathbb{E}\left[#1\right]}		

\newcommand{\Card}{\operatorname{Card}}			
\newcommand{\supp}{\mathrm{supp}}			

\newcommand{\ds}{\displaystyle}
\newcommand{\ie}{{\it i.e.}}
\newcommand{\as}{{\it a.s.}}
\renewcommand{\ae}{{\it a.e.}}

\def\ds{\displaystyle}

\begin{document}

\begin{frontmatter}

\title{Boundary value for a nonlinear transport equation emerging from a stochastic coagulation-fragmentation type model \tnoteref{label1}}
\tnotetext[label1]{This work has been supported by ANR-14-CE25-0003 (Julien Deschamps), FONDECYT Grant no. 3130318 (Erwan Hingant) and INRA (Romain Yvinec)}

\author[JD]{Julien Deschamps}
\ead{deschamps@dima.unige.it}

\author[EH]{Erwan Hingant}
\ead{ehingant{@}ci2ma.udec.cl}

\author[RY]{Romain Yvinec}
\ead{romain.yvinec@tours.inra.fr}

\address[JD]{DIMA, Universit\`a degli Studi di Genova, Italy.}
\address[EH]{CI\textsuperscript{2}MA, Universidad de Concepci\'on, Chile.}
\address[RY]{BIOS group, INRA, UMR85, Unit\'e Physiologie de la Reproduction et des Comportements, F-37380 Nouzilly, France;
CNRS, UMR7247, F-37380 Nouzilly, France; Universit\'e François Rabelais, 37041 Tours, France; IFCE, Nouzilly,
F-37380 France.}

\begin{abstract}
We investigate the connection between two classical models of phase transition phenomena, the (discrete size) stochastic Becker-D\"oring, a continous time Markov chain model, and the (continuous size) deterministic Lifshitz-Slyozov model, a nonlinear transport partial differential equation. For general coefficients and initial data, we introduce a scaling parameter and prove that the empirical measure associated to the stochastic Becker-D\"oring system converges in law to the weak solution of the Lifshitz-Slyozov equation when the parameter goes to 0. Contrary to previous studies, we use a weak topology that includes the boundary of the state space (\ie\ the size $x=0$) allowing us to rigorously derive a boundary value for the Lifshitz-Slyozov model in the case of incoming characteristics. The condition reads $\lim_{x\to 0} (a(x)u(t)-b(x))f(t,x) = \alpha u(t)^2$ where $f$ is the volume distribution function, solution of the Lifshitz-Slyozov equation, $a$ and $b$ the aggregation and fragmentation rates, $u$ the concentration of free particles and $\alpha$ a nucleation constant emerging from the microscopic model. It is the main novelty of this work and it answers to a question that has been conjectured or suggested by both mathematicians and physicists. We emphasize that this boundary value depends on a particular scaling (as opposed to a modeling choice) and is the result of a separation of time scale and an averaging of fast (fluctuating) variables.  
\end{abstract}


\end{frontmatter}
\section{Introduction}\label{intro}

This papers addresses the mathematical connection between two classical models of phase transition phenomena describing different stages of cluster growth. 

The first one is the Stochastic Becker-D\"oring (SBD) model \cite{EBELING,BHATT,DLC,Yvinec2012}, representing the microscopic stages. This model is a subclass of general finite-particle stochastic coagulation-fragmentation models \cite{Aldous1999} and it corresponds to the continous time Markov chain version of the so-called (deterministic) Becker-D\"oring (BD) model \cite{Becker35,Penrose2001,Wattis2008}. In this model, clusters of particles may increase or decrease their size (number of particles in the cluster) one-by-one by capturing (aggregation process) or shedding (fragmentation process) one particle, according to the set of chemical reactions

\begin{equation}\label{eq:chemreact_intro}
  \ds C_1 + C_i \ds \xrightleftharpoons[b_{i+1}]{a_{i}}  \ds C_{i+1} \,, \quad i\geq 1\,,
\end{equation}
where $C_i$ stands for the clusters consisting of $i$ particles and $C_1$ the free particle. Here, the coefficients $a_i>0$ and $b_{i+1}> 0$ denote respectively the rates of aggregation and the rates of fragmentation. The SBD model is defined as a Markov chain on a finite subset of a lattice. Choose a (possibly random, but almost surely finite) parameter $M\geq 2$ that gives the total number of particles in the system (number of free particles and number of particles in the clusters). Since the mass of each particle can be fixed at $1$ without loss of generality, this quantity is also called the total mass of the system. The state space of the process is given by
\begin{equation*}
\Ec := \Big\{ (C_i)_{i \geq 1} \subset \Nb \, : \,\sum_{i\geq 1} i C_i = M \Big\}\,.
\end{equation*}
The  reactions of aggregation and fragmentation in \eqref{eq:chemreact_intro} are the  processes between clusters of two successive sizes, with rate (intensity of a Poisson process) $a_1C_1(C_1-1)$ for $i=1$, $a_i C_1 C_{i}$, $i\geq 2$, and $b_{i+1}C_{i+1}$, $i \geq 1$, respectively (law of mass action). This set of kinetics reactions \eqref{eq:chemreact_intro} completely defines a well-posed model. Indeed, when the initial condition $(C_i(0))_{i\geq1}$ belongs to $\Ec$, it is then trivial to see that $(C_i(t))_{i\geq1}$ can be re-written as a Markov chain in a finite state space ($\Card(\Ec) <\infty$), for which existence for all times is guaranteed.

\medskip

The second model is the Lifshitz-Slyozov (LS) model \cite{LS61}, and describes the cluster growth at a macroscopic scale. Accordingly, the size of the clusters are represented by a continuously varying variable $x >0$. The LS model consists in a partial differential equation (of nonlinear transport type) for the time evolution of the volume distribution function $f(t,x)$ of clusters of size $x$, together with an equation stating the conservation of matter,

\begin{equation}\label{eq:LS}
 \begin{array}{ll}
  \ds \partial_t f (t,x)+ \partial_x [ (a(x)u(t) - b(x))f(t,x) ] =0\,, & t\geq0\,, \ x>0\,, \\[1.5em] 
  \ds u(t)  +  \int_0^\infty xf(t,x) = \text{const.} \, , & t\geq 0\,, 
 \end{array}
\end{equation}
where $a(x)$ and $b(x)$ are two functions of the size, respectively for the aggregation and fragmentation rates. Note that in such a model, $u(t)$ plays the \textit{analog} role of the concentration of free particles $C_1(t)$ in the SBD model. Under classical conditions on $a$ and $b$, it is known that system \eqref{eq:LS} is well-defined when the flux is pointing outwards of the domain, namely if $a(0)u(t) - b(0) < 0$, otherwise it lacks a proper \textit{boundary condition} at $x=0$. For theoretical studies on the well-posedness and long-time behavior of the LS model, we refer the interested reader to \cite{Laurenccot2002,Collet2002a,Niethammer2008}.

\medskip

The two (discrete-size) BD and (continuous-size) LS models have been rigorously connected within the context of deterministic models. Two main approaches are used. One can consider the large time behavior of the BD model, and relates the dynamics of large clusters to solutions of various version of LS equations. It is the so-called theory of Ostwald ripening, see \cite{Penrose1997,Niethammer2005,Velazquez1998}. The other approach identifies the ``macroscopic limit'' of the BD system, and considers an initial condition with a large excess of particles. Then, an appropriate re-scaling of this initial condition together with the time and the rate functions leads to solutions of the LS equation, see \cite{Collet2002,Laurencot2002a,doumic}. In the two last aforementioned works, the authors introduce a suitable scaling parameter and are able to prove that the solution of the (deterministic) BD model converges (in an appropriate sense) towards a solution of the LS equations \eqref{eq:LS}, as the scaling parameter converges towards $0$. These results were restricted to weak convergence in a \textit{vague} topology, that is against test functions that vanish at the boundary $x=0$ of the physical state space (see also Theorem \ref{thm:LS_convvague} below). As such, the convergence results was restricted to cases where the problem does not require any boundary value in order to uniquely define a solution of the limit system \eqref{eq:LS}, \textit{i.e.} to cases where $a(0)u(t)-b(0)<0$.

\medskip

\noindent The aim of this paper is to extend the previous (deterministic) results obtained in \cite{Collet2002,Laurencot2002a} in two directions.

\smallskip

{\bf 1.} Instead of starting with the deterministic version of the BD model, we use the stochastic (SBD) version to connect both models, inspired from \cite{Fournier,Fourniera} on the Marcus-Lushnikov model and \cite{champagnat08} on structured population models. In the stochastic context developed here, we can expect that second-order approximation and large deviation results to be of a qualitative different nature than the corresponding deterministic context. Hence our approach has its interest in its own, see Discussion. Moreover, it seems it is the first time a rigorous link from discrete-size stochastic to continuous-size deterministic coagulation fragmentation model is proposed.

\smallskip

{\bf 2.}  The most important part. We extend the previous results because we are able to rigorously identify, for general scaling, a boundary-value for the LS equation in the case of incoming flux when $u(t) > \lim_{x\to0} b(x)/a(x) \in [0,+\infty]$. This is the main novelty of our results. In particular, we obtain a convergence result (see Theorem \ref{thm:LS_convweak}) towards a (weak measure) solution of the LS equation \eqref{eq:LS} with a flux condition corresponding to
\begin{equation}\label{eq:boundary_flux}
 \lim_{x\to 0} \ (a(x)u(t)-b(x))f(t,x) = \alpha u(t)^2\,,
\end{equation}
where $\alpha$ is explicitly derived from the microscopic rates of the initial model together with proper rescaling. Such boundary conditions were conjectured {\it e.g.} in  \cite{Collet2002,Prigent2012}  but never proved. Historically, there was no need of a boundary-value in  LS since the problem was well-posed under physical assumptions (when small clusters tend to fragment). But, recent applications in Biology have raised this problem to include nucleation in this equation, for instance in \cite{Prigent2012,Helal2013}. Furthermore, such results raise the possibility to obtain quantitative approximation of time-dependent solutions (or related quantities, such as first passage time) of the SBD with the help of the limit macroscopic model (\ref{eq:LS}-\ref{eq:boundary_flux}), see Discussion. 

\smallskip

Finally, we would mention the originality of our work resides in the proof too. Indeed, to identify the boundary we carefully introduce particular measure spaces and their topologies. Then, we adapt to our context the tools developed in \cite{Kurtz1992} about averaging to obtain the limit of some fast (fluctuating) variables. This averaging procedure yields the identification of the boundary condition. To the best of our knowledge, such a strategy to rigorously derive a boundary condition for a transport equation seems to be new.

\medskip 

\paragraph{\textbf{Organization of the paper}}We start by introducing the measure-valued SBD model in Section \ref{sec:main} together with its rescaled generator. This section presents in a concise manner our main results in Theorems \ref{thm:SBD_rescale_limit}, \ref{thm:LS_convvague} and \ref{thm:LS_convweak}. Stochastic equations, martingale properties and the scaling laws used to rescale the SBD model are detailed in Section \ref{sec:martingale}. We go on with technical results on moment estimates and tightness properties in Section \ref{sec:technical} in order to prepare the proof of the main results. We emphasize in this section the introduction of a particular (occupation) measure containing the information on the boundary value. In Section \ref{sec:proof} we prove our main three theorems. We then conclude by a Discussion in Section \ref{sec:disc} with literature comparison, possible applications and future directions of our work, illustrated with the help of numerical simulations.

\paragraph{\textbf{Notations}} For the remainder we introduce few classical notations we will use for sake of clarity. First, $\mathcal C$ denotes the space of continuous functions. Similarly, $\Cc_b$, $\Cc_c$ and $\Cc_0$ are the spaces of continuous functions which are, respectively, bounded, compactly supported and vanishing at boundary (seen as a closure of $\Cc_c$). We denote by $\mathcal C^{k}$ the functions having $k$ continuous derivatives (up to $k=\infty$). Similarly for the other spaces the $k$ derivatives have the same regularity. 
For a Polish space $E$, we denote by $\Mc(E)$ the set of non-negative Radon measures on $E$, $\Mc_b(E)$ the set of non-negative and finite Radon measures on $E$ and $\Pc(E)$ the set of probability measures.
The convergence in $\Pc(E)$ of a sequence $X^n$ of $E$-valued random variable has to be understood as the classical convergence in law or distribution of random variables. For any $\nu \in \Mc_b(E)$ and $\vphi$ a real-valued measurable function on $E$, we write  

\[\lb \nu,\vphi\rb_{E} = \int_E \vphi(x)\nu(dx)\,.\]
When no doubt remains on the measurable space $E$, we will simply write $\lb \nu,\vphi\rb$ instead of $\lb \nu,\vphi\rb_{E}$.

\section{Empirical measure, re-scaled process and main results}\label{sec:main}

The SBD model introduced in the previous section could be studied using classical tools from Markov chains, such as stochastic equations, Chapman-Kolmogorov equations, first-passage time analysis, \emph{etc}. As our objective is to investigate the limit as the total mass $M\to \infty$ (large numbers) and to recover a weak form of a deterministic partial differential equation, it is preferable to use a measure-valued stochastic process approach. The advantage is to get a fixed state space while performing the limit $M\to\infty$. To that, we consider the set 

\[\Mc_\delta  :=  \left\{ \sum_{i=1}^n \delta_{x_i} \, : \,  n\geq 0,  \, (x_1,\ldots,x_n) \in \Nb^n, \ x_i \geq 2, \ \forall i  \right\}\subset \Mc_b(\Rb_+)\,.\]
We represent the population of clusters, with the following measure at time $t\geq0$

\begin{equation}\label{eq:empirical_measure}
\mu_t = \sum_{i\geq 2} C_i(t)\,\delta_{i} \in \Mc_\delta\,.
\end{equation}
Note that the number of clusters for a given size $i\geq 2$ is $C_i(t)=\mu_t(\{i\})$. This point of view defines $(\mu_t)_{t\geq 0}$ as a measure-valued stochastic process that entirely encodes the information of the system. Moreover, since the free particle quantity $C_1$ is not included in the measure $\mu_t$, we recover it through the balance of mass, which reads now (with $\Id$ the identity function)

\[ C_1(t) +  \lb \mu_t ,\mathrm{Id} \rb = \sum_{i\geq 1} i C_i(t) =  M\,, \quad t\geq 0\,. \]

As we will see in Section \ref{sec:martingale}, the infinitesimal generator $\Lc$ of the measure-valued stochastic process $(\mu_t)_{t\geq0}$ is given, for all $\nu \in \Mc_\delta$ and for all locally bounded measurable function $\psi$ from $\Mc_b$ to $\Rb$, by
\begin{multline}\label{eq:gene_BD_stoch}
 \mathcal L \psi(\nu) =    \left[ \psi(\nu+\delta_{2}) - \psi(\nu)\right] a_1 C(C-1) \vphantom{\sum_{i\geq k}} + \left[ \psi(\nu-\delta_{2}) - \psi(\nu)\right] b_2 \nu(\{2\})\vphantom{\sum_{i\geq k}} \\
  +  \sum_{i\geq 2}  \left[ \psi(\nu+\delta_{i+1}-\delta_{i}) - \psi(\nu)\right]a_i C \nu(\{i\})  +  \sum_{i\geq 3}  \left[ \psi(\nu-\delta_i+\delta_{i-1}) - \psi(\nu)\right]b_i \nu(\{i\}) \, ,
\end{multline}
where $C = M - \lb \nu,\mathrm{Id} \rb$.  Comparing to the chemical reactions in \eqref{eq:chemreact_intro}, the first two terms in Equation \eqref{eq:gene_BD_stoch} correspond to the reactions for $i=1$ in \eqref{eq:chemreact_intro}, while the two sums in the generator correspond to the reactions for $i\geq 2$ in \eqref{eq:chemreact_intro}. Further properties are detailed in Section \ref{sec:martingale}.

Our main object will be a rescaled version of this generator. We introduce a small parameter $\veps>0$ and we study the limit as $\veps \to 0$ of a rescaled measure-valued SBD process $(\mu^\veps_t)_{t\geq 0}$. This process depends on aggregation and fragmentation rates, $a^\veps$ and $b^\veps$ that are now defined as a function on $\Rb_+$. It also depends on two parameters $\alpha^\veps$ for the nucleation rate (formation of a cluster of size $2$) and $\beta^\veps$ for the de-nucleation (fragmentation of a cluster of size $2$). We give the following definition.
\begin{definition} \label{def:rescaleSBD}
 Consider an initial state $\mu^\veps_{\rm in}$ that belongs to 
\[ \Mc_\delta^\veps:= \Big\{ \nu \in \Xc\, : \, \nu = \sum_{i=1}^n \veps \delta_{\veps x_i} \,, \ (x_1,\ldots,x_n)\in \Nb, \ x_i\geq 2, \ \forall i \Big\}\,,\]
and $u_{\rm in}^\veps \in \veps\Nb$. We define $ m^\veps := u^\veps_{\rm in} + \lb \mu^\veps_{\rm in},\Id \rb$. The {\it rescaled measure-valued Stochastic Becker-D\"oring process} $(\mu^\veps_t)_{t\geq 0} \in \Dc(\Rb_+,w-\Xc) $ is a Markov $\Xc$-valued c\`adl\`ag process, where 
\[ \Xc := \Big\{ \nu \in\Mc_b(\Rb_+)\,: \, \lb \nu,\Id \rb < +\infty \Big\} \,,\]
taking its values in  $\Mc_\delta^\veps$ and having its infinitesimal generator given, for all $\nu \in \Mc_\delta^\veps$ and for all locally bounded measurable function $\psi$ from $\Xc$ to $\Rb$, by
\begin{multline}\label{eq:gene_BD_stoch_rescale}
 \mathcal L^\veps \psi(\nu)  =   \ds \frac{\psi(\nu+ \veps \delta_{2\veps}) - \psi(\nu)}{\veps} \alpha^\veps c(c-\veps^2) \vphantom{\sum_{i\geq k}}\
 + \frac{ \psi(\nu-\veps \delta_{2\veps}) - \psi(\nu)}{\veps} \beta^\veps \nu (\{2\veps\}) \vphantom{\sum_{i\geq k}} \\
  +  \int_{2\veps}^{+\infty}   \frac{\psi(\nu+ \veps \delta_{x+\veps}-\veps \delta_{x}) - \psi(\nu)}{\veps^{2}}  a^\veps(x) c \nu(dx)  \vphantom{\sum_{i\geq k}} +  \int_{3\veps}^{+\infty} \frac{\psi(\nu-\veps \delta_{x}+\veps \delta_{x-\veps}) - \psi(\nu)}{\veps^{2}}  b^\veps(x) \nu(dx) \vphantom{\sum_{i\geq k}}\, ,
\end{multline}
where $c=m^\veps - \lb \nu,\Id \rb$.
\end{definition}
\begin{remark}
 Here, the Polish space $\Xc$ is equipped with the weak topology as described in \ref{app:space_X} further denoted by $w-\Xc$ (or alternatively $(\Xc,w)$). The space $\Dc(\Rb_+,w-\Xc)$ denotes the $\Xc$-valued c\`adl\`ag function and is equipped with the Skorohod topology (see \cite{EthierKurtz} for more details).
\end{remark}

For each $\veps>0$, this process in Definition \ref{def:rescaleSBD} is well-defined and unique since, conditionally to $\mu^\veps_{\rm in}$ and $u^\veps_{\rm in}$, it is a continuous-time Markov chain with finite state space. In Section \ref{sec:martingale} we introduce the scaling used to rigorously derive the generator $\Lc^\veps$ from the original one $\Lc$, but it should be clear by virtue of the state space that we made a scaling of number and size all together. Our first limit theorem below depends on some assumptions we detail here.

\begin{hyp}[Convergence of rate functions and parameters] \label{hyp:cv_fct}
Assume that there exist two non-negative constants $\alpha$ and $\beta$, and two continuous non-negative functions $a$ and $b$ such that
\begin{align}
 &\{\alpha^\veps\} \text{ converges towards } \alpha \,. \tag{H1} \label{H1}\\
 &\{\beta^\veps\} \text{ converges towards } \beta \,. \label{H2} \tag{H2} \\
 &\{a^\veps\} \text{ converges uniformly on any compact set of } \Rb_+ \text{ towards } a \text{ and } \nonumber \\
 & \hspace{3em}\exists K_a>0 \text{ s.t. } a^\veps(x) \leq K_a (1+x), \ \forall x\in \Rb_+ \text{ and } \forall \veps>0\,. \label{H3} \tag{H3} \\
 &\{b^\veps\} \text{ converges uniformly on any compact set of } \Rb_+ \text{ towards } b \text{ and } \nonumber \\
 & \hspace{3em}\exists K_b>0 \text{ s.t. } b^\veps(x) \leq K_b (1+x), \ \forall x\in \Rb_+ \text{ and } \forall \veps>0\,. \label{H4} \tag{H4} 
\end{align}

\end{hyp}

\begin{remark}
Hypotheses \eqref{H3} to \eqref{H4} entail that for all $x\geq 0$, $a(x) \leq K_a (1+x)$, and $b(x) \leq K_b (1+x)$.
\end{remark}
To determine the boundary value, we need to define the occupation measure (see \cite{Kurtz1992}) of  the evaluation at a given finite size of $\mu^\veps$. For that we define the sequence $p^\veps_t=(p^\veps_{n,t})_{n\geq0}$ by  $p_{n,t}^\veps=\mu_t^\veps(\{\veps(n+2)\})$ for all $n\geq0$. The sequence $p_t^\veps$ clearly belongs to $\ell^+_1$, the non-negative cone of the summable sequences. The occupation measure is then defined by 
\begin{equation}\label{def:Gamma}
 \Gamma^\veps(A\times B) = \int_A \indic{\{p_s^\veps\in B\}} \, ds \,,
\end{equation}
for all Borel set $A$ of $\Rb_+$ and Borel set $B$ of $\ell_1^+$ equipped with the vague topology (the convergence against sequences that vanish at infinity). This measure belongs to the subspace $\Yc$ of non-negative measures on $\Rb_+\times \ell_1^+$ given by
\begin{equation*}
\Yc:= \Bigg{\{} \Theta \in  \Mc(\Rb_+\times \ell_1^+)\, : \, \forall t\geq 0,\   \Theta([0,t]\times \ell_1^+) = t\,,
\  \int_{[0,t]\times \ell_1^+} (1+\lVert q \rVert_{\ell_1}) \, \Theta(ds\times dq) < +\infty \Bigg{\}} \,. 
\end{equation*}

We are now in position to state our first result.

\begin{theorem} \label{thm:SBD_rescale_limit}
Suppose that Assumption \ref{hyp:cv_fct} holds, $\{\mu_{\rm in}^\veps\}$ converges in law towards a deterministic measure $\mu_{\rm in}$ by staying \as\ in a weakly compact set $\Kc$ of $\Xc$, that is, we have $\Pb(\mu^\veps_{\rm in} \in \Kc) = 1$ for all $\veps>0$,  and $u^\veps_{\rm in}$ converges towards a deterministic $u_{\rm in}$ in $\Pc(\Rb_+)$.

Then, the SBD process $\{\mu^\veps\}$ and the occupation measure $\{\Gamma^\veps\}$ converge along an appropriate subsequence as $\veps \to 0$, respectively, to $\mu$ in $\Pc(\Dc(\Rb_+,w-\Xc))$ and $\Gamma$ in $\Pc(w^\#-\Yc)$. The limit $\mu$ belongs to $\Cc(\Rb_+,w-\Xc)$ and we have, for all $\vphi \in \Cc^1_b(\Rb_+)$ and $t\geq 0$, \as
 \begin{multline}\label{eq:weak_LS_vague1}
  \brak{\mu_t}{\vphi}=\brak{\mu_{\rm in}}{\vphi}  + \int_0^t \int_{0}^{\infty} \vphi'(x)(a(x)u(s)-b(x))\mu_s(dx)\,ds+ \vphi(0)  \int_0^t  \alpha u(s)^2 \,ds \\
  -   \vphi(0) \int_{[0,t]\times\ell_1^+} \beta q_0 \Gamma(ds\times dq) + \vphi'(0) \int_{[0,t]\times\ell_1^+} b(0) q_{0} \Gamma(ds\times dq), 
 \end{multline}
 where $u(t) + \brak{\mu_t}{\mathrm{Id}} = m$ given by $m:=u_{\rm in} + \brak{\mu_{\rm in}}{\mathrm{Id}}$ and $q_0$ is the first component of the variable $q\in \ell_1^+$.
\end{theorem}


\begin{remark}
Here, the Polish space $w^\#-\Yc$  (or alternatively $(\Yc,w^\#)$ in the remainder) denotes the space $\Yc$ equipped with the weak$^\#$ topology as described in \ref{app:space_Y}. 
\end{remark}

\begin{remark}
The existence of a weakly compact set $\Kc$ of $\Xc$ on the initial measures will be useful for tightness criteria (see Section \ref{sec:technical}).
\end{remark}

In Equation \eqref{eq:weak_LS_vague1}, the second term on the right hand side is the classical drift (in weak form) of the Lifshitz-Slyozov equation. Moreover, the terms involving the limit occupation measure $\Gamma$ contribute to the boundary value. A simple computation, taking $\vphi= \indic{}$, shows that the terms in $\vphi(0)$ account for the number of clusters created with critical size $0$ (nucleation). While taking $ \vphi=\Id$, the term in $\vphi'(0)$ gives a mass to the clusters of size $0$.


As a direct consequence of this theorem, taking $\vphi\in C^1_c(\Rb_+^*)$, we recover a measure solution of the classical Lifshitz-Slyozov equation in a sense called {\it vague}.

\begin{theorem}\label{thm:LS_convvague}
Under the same hypotheses as Theorem \ref{thm:SBD_rescale_limit}, the limit $\mu$ is a vague solution of the Lifshitz-Slyozov equation, that is \as~for all $\vphi \in \Cc^1_c(\Rb_+^*)$ and $t\geq0$

\begin{equation} \label{eq:vague_LS}
  \brak{\mu_t}{\vphi}=\brak{\mu_{\rm in}}{\vphi}  + \int_0^t \int_{0}^{\infty} \vphi'(x)(a(x)u(s)-b(x))\mu_s(dx)ds\,,
\end{equation}
where $u(t) + \brak{\mu_t}{\mathrm{Id}} = m$.
\end{theorem}

\noindent This equation is known to be well-posed (uniqueness) in the case of ``outgoing characteristics''. Indeed, this theorem is limited by the fact that the test functions do not account for the boundary value in $0$. Thanks to the result given by Collet and Goudon in \cite[Theorem 3]{Collet2000} it readily follows:

\begin{corollary}
In addition to hypotheses of Theorem \ref{thm:LS_convvague}, assume that $a$ and $b$ belong to $C^1(\Rb_+)$. For any $T>0$ such that the limit $u$ satisfies  \[a(0)\sup_{t\in[0,T]}u(t)-b(0) \leq 0\,,\]  Equation \eqref{eq:vague_LS} has a unique solution $\mu$ in $\Cc([0,T],w-\Xc)$, hence the whole sequence $\{\mu^\veps\}$ converges in  $\Pc(\Dc([0,T],w-\Xc))$ to $\mu$. 
\end{corollary}

This corollary does not include cases where $a$ and $b$ behave as a power law ($x\to x^\eta$) with $\eta$ in $(0,1)$, as it is usual. Note that a better result of uniqueness is available in \cite{Laurencot2001} for density solution. Of course, we are interested in the case of ``incoming characteristics'' when a boundary condition is necessary for the well-posedness. Thus, we need to identify what is $\Gamma$. In order to do so, we need to know the behavior of the rate functions $a$ and $b$ near $0$. More precisely, we suppose that the limit functions $a$ and $b$ behave as a power-law function near $0$ and similarly for the approximations $a^{\veps}$ and $b^{\veps}$.

\begin{hyp}[Behavior of the rate functions near $0$] \label{hyp:coef_BD}

We suppose there exist $r_a, \ r_b \geq 0$ with $\min(r_a,r_b)<1$, and $\overline{a}, \ \overline{b}>0$ such that
\begin{equation}\label{cv_a,b} \tag{H5}
 a(x) \sim_{0^+} \overline{a}x^{r_a}\, \text{ and }\ b(x) \sim_{0^+} \overline{b}x^{r_b}\,,
\end{equation}
and that,
\begin{equation}\label{H6}\tag{H6}
 a^{\veps}(\veps i) = a(\veps i) + o(\veps^{r_a})\,, i\geq 2\,, \text{ and } b^{\veps}(\veps i) = b(\veps i) + o(\veps^{r_b})\,, i\geq 3\,.
 \end{equation}


\end{hyp}
%

\noindent Before stating the second theorem we introduce a critical threshold which will be debated below, namely

\begin{equation} \label{def:rho}
  \rho := \lim_{x\to 0} \frac{b(x)}{a(x)} \in [0,+\infty]\,.
\end{equation}
The result reads:

\begin{theorem}\label{thm:LS_convweak}
In addition to hypotheses of Theorem \ref{thm:LS_convvague}, suppose that Assumption \ref{hyp:coef_BD} holds. Then, on any time interval $[t_0,t_1]$ such that the limit $u(t)>\rho$ for all $t\in[t_0,t_1]$, the limit measure $\Gamma$ vanishes and $\mu$ is a weak solution of the Lifshitz-Slyozov equation, that is for all $\vphi\in\Cc_b^1(\Rb_+)$ and $t\in[t_0,t_1]$, \as
 
 \begin{equation}\label{eq:weak_LS}
  \brak{\mu_t}{\vphi}=\brak{\mu_{t_0}}{\vphi}  + \int_{t_0}^t \int_{0}^{\infty} \vphi'(x)(a(x)u(s)-b(x))\mu_s(dx)ds + \vphi(0) \int_{t_0}^t  \alpha u(s)^2 ds\,, 
 \end{equation}
 with $u(t) + \brak{\mu_t}{\mathrm{Id}} = m$.
\end{theorem}

Let us discuss the different scenarios. If $0 \leq  r_a<r_b$, the aggregation term is stronger than the fragmentation near $x=0$, and small size clusters can growth for all time since $\rho = 0$. If $0 \leq r_a=r_b$, it is a limit case and $\rho = \overline b / \overline a$. The nucleation occurs when enough free particles is supplied. Note that in the case $\rho>m$, we always have $u(t)\leq m<\rho$. While if $0 \leq r_b < r_a$,  the fragmentation is stronger than the aggregation in $0$ and Theorem \ref{thm:LS_convweak} is nothing compared to Theorem \ref{thm:LS_convvague} since $\rho = +\infty$.  The case $\min(r_a,r_b)\geq 1$ is related, either to an outgoing case or to a case where no boundary condition is needed. The latter corresponds to the case where small size clusters cannot growth in finite time and is thus not investigating here.

\noindent The uniqueness of this latter theorem is left since the measure formulation together with the regularity of the coefficients near $0$ make the problem difficult to treat. The reader interested in this problem and its uniqueness should probably refer to \cite{Boyer2005}.
%

\noindent Finally, we mention that the boundary value can be interpreted as a flux condition in the case of a density solution, that is $\mu_t = f(t,x)dx$, and it reads as mentioned in Equation \eqref{eq:boundary_flux}.

\section{Stochastic equations, martingale properties and rescaling} \label{sec:martingale}

In this section we first introduce the probabilistic objects we use to define the stochastic equation for the measure-valued SBD process (Poisson processes). Both the stochastic equation and the generator are used later to obtain accurate moment estimates. Then we introduce the scaling which allows us to define the process given in Definition \ref{def:rescaleSBD}. And finally we state some properties of this rescaled process.

\subsection{The original SBD process}

We define below, first the probabilistic objects we use and then the stochastic differential equation satisfied by the (original) empirical measure \eqref{eq:empirical_measure}. To make some estimates easier later on, we use four distinct Poisson measures to classify the different reactions in Equation \eqref{eq:chemreact_intro}.

\begin{definition}[Probabilistic objects] \label{def:objects}
Let $(\Omega,\Fc, \Pb)$ a sufficiently large probability space. $\esp{\cdot}$ denotes the expectation. We define on this space four independent random Poisson point measures 
\begin{itemize}
 \item[i)] The nucleation Poisson point measure $Q_1(dt,dy)$ on $\Rb_+\times\Rb_+$ with intensity \[ \esp{Q_1(dt,dy)} = dtdy  \,. \]
 \item[ii)] The de-nucleation Poisson point measure $Q_2(dt,dy)$ on $\Rb_+\times\Rb_+$ with intensity \[ \esp{Q_2(dt,dy)} = dtdy  \,. \]
 \item[iii)] The aggregation Poisson point measure $Q_3(dt,dy,di)$ on $\Rb_+\times\Rb_+\times \Nb\backslash\{0,1\} $ with intensity \[ \esp{Q_3(dt,dy,di)} = dtdy \ \#(di) \,. \]
 \item[iv)] The fragmentation Poisson point measure $Q_4(dt,dy,di)$ on  $\Rb_+\times\Rb_+\times \Nb\backslash\{0,1,2\}$ with intensity \[\esp{Q_4(dt,dy,di)} = dtdy\ \#(di) \,. \]
 \end{itemize}
where $dt$ and $dy$ are Lebesgue measures on $\Rb^+$, and $\#(di)$ is the counting measure on $\Nb$. Moreover, we define two more independent (from the above) random elements
\begin{itemize}
 \item[v)] The initial distribution $\mu_{\rm in}$ is a $\Xc$-valued random variable such that \as~ $\mu_{\rm in}$ belongs to $\Mc_\delta$.
 \item[vi)] The initial quantity of free particles $C_{1, \rm in}$ is a $\Nb$-valued random variable (\as~finite).
\end{itemize}
Finally, we define the canonical filtration $(\mathcal F_t)_{t\geq0}$ associated to the Poisson point measures such that $\mu_{in}$ and $C_{1, \rm in}$ are measurable too.
\end{definition}
Now we give a definition-proposition of the measure formulation of the SBD model.

\begin{definition}[SBD process] \label{def:eds_BD}

Assume the probabilistic objects of Definition \ref{def:objects} given. A {\it measure-valued Stochastic Becker-D\"oring process} is the unique $\Xc$-valued c\`adl\`ag process $\mu=(\mu_t)_{t\geq 0}$ taking its values in $\Mc_\delta$ and satisfying, \as~and for all $t\geq 0$,
\begin{equation} \label{eq:eds}
 \begin{array}{rl}
 \ds \mu_t  =  \mu_{\rm in}  & \ds +  \int_0^t \int_{\Rb_+} \delta_{2} \indic{ \Enbrace{y\leq  a_1C_1(s^-)(C_1(s^-)-1)} } \, Q_1(ds,dy)  \\
 		     & \ds - \int_0^t \int_{\Rb_+} \delta_{2} \indic{ \Enbrace{y\leq   b_2 \mu_{s^-}(\{2\}) } } \, Q_2(ds,dy)  \\
		     & \ds + \int_0^t \int_{\Rb_+\times \Nb\backslash\{0,1\} } \left(\delta_{i+1}-\delta_{i}\right) \indic{ \Enbrace{y\leq  a_i C_1(s^-) \mu_{s^-}(\{i\}) } } \, Q_3(ds,dy,di) \\
		     & \ds - \int_0^t \int_{\Rb_+ \times \Nb\backslash\{0,1,2\} } \left(\delta_i-\delta_{i-1}\right) \indic{ \Enbrace{y\leq  b_i \mu_{s^-}(\{i\}) }} \, Q_4(ds,dy,di)\,, 
 \end{array}
\end{equation}
with the balance law given, also \as~for all $t\geq 0$, by
\begin{equation} \label{eq:constraint}
C_1(t) + \brak{\mu_t}{\mathrm{Id}} = M \,,
\end{equation}
where $M$ is given by the initial state.
\end{definition}

\begin{remark}
 The total mass $M$ is a random element defined by $M := C_{1,\rm in} + \brak{\mu_{\rm in}}{\mathrm{Id}}$ and is \as~finite. Moreover, the SBD process satisfies, \as~for all $t\geq 0$, $\lb \mu_t,\indic{} \rb \leq \frac{M}{2}$.
We emphasize that this stochastic process is still evolving in a finite state space that is a subset of 
$\{\nu \in \Mc_{\delta} \, : \, \sum_{i\geq 2} i \nu(\{i\}) \leq  M\}$. Hence all properties on non-explosion, generator and martingale properties are trivial. In particular, the generator of the solution $\mu$ of Equation \eqref{eq:eds} is given by \eqref{eq:gene_BD_stoch}.
\end{remark}

\subsection{Definition of the scaling}\label{sec:rescale}

The classical approach to operate a scaling is to write the equations in a dimensionless form. We follow \cite{Collet2002} and introduce the following characteristic values:

\begin{itemize}
 \item $\overline T$ : characteristic time,
 \item $\overline C_1$ : characteristic value for the free particle number $C_1$ ,
 \item $\overline C$ : characteristic value for the cluster number $\mu(\{i\})$, for $i \geq 2$,
 \item $\overline A_1$ : characteristic value for the first aggregation coefficient $a_1$,
 \item $\overline B_2$ : characteristic value for the first fragmentation coefficient $b_2$,
 \item $\overline A$: characteristic value for the aggregation coefficients $a_i$, $i\geq 2$,
 \item $\overline B$: characteristic value for the fragmentation coefficients $b_i$, $i\geq3$, 
 \item $\overline M_c$ : characteristic value for the total mass.
\end{itemize}
Thus, the dimensionless quantities are 
\[\tau = t/\overline T \,, \quad \tilde m = M/\overline M_c\,, \quad \tilde u(\tau) = C_1(\tau T)/\overline C_1\,, \]
for all $i\geq 2$,
\[\tilde a_i  =  a_i/\overline A\,, \quad  \tilde b_{i+1}  =   b_{i+1}/\overline B\,, \]
and the particular scaling at the boundary (we use different letters to emphasize this point):
\[\tilde \alpha  :=  a_1/ \overline A_1\,, \quad \tilde \beta :=   b_2/ \overline B_2\,. \]
We introduce the scaling parameter $\veps>0$ for the size of the clusters, and our model is derived from the following choices of relation:
\[ \overline C = \frac 1 \veps\,, \ \overline C/ \overline C_1 = \veps \,,\quad   \overline A \, \overline C_1\overline  T = \overline B \, \overline T =  \frac{1}{\veps}\,,\quad   \overline M_c / \overline C_1 = 1,\]
and 
\[  \overline A_1 = \veps^{2} \overline A \,, \quad  \overline B_2 = \veps \overline B\,. \]
The reader interested in a physical justification of this scaling can refer to \cite{Collet2002}. The only difference we made here from \cite{Collet2002} is a slowdown of the first fragmentation rate (de-nucleation). This turns out to be equivalent to assume that (asymptotically) the nucleation is irreversible. Finally, we rescale the measure (with an explicit dependence on $\veps$) by
\begin{equation} \label{eq:def_rescale_measure}
\tilde \mu_\tau^\veps = \sum_{i\geq 2}  \frac{\mu_{\tau \overline T}(\{i\})}{\overline C} \delta_{\veps i}\,.
\end{equation}
We can now write the SDE on the rescaled measure, dropping tilde but mentioning the explicit dependence on $\veps$ for all the coefficients: \as~for all $t \geq 0$,
\begin{equation} \label{eq:rescaled_eds}
 \begin{array}{rl}
 \ds  \mu_t^\veps =   \mu^\veps_{\rm in}  & \ds +  \int_0^t \int_{\Rb_+} \veps \delta_{2\veps} \indic{ \Enbrace{y\leq  \alpha^\veps u^\veps(s^-)(u^\veps(s^-)-\veps^2)/(\veps \overline T)} } \, Q_1(\overline T ds,dy)  \\
 		     & \ds - \int_0^t \int_{\Rb_+} \veps \delta_{2\veps} \indic{ \Enbrace{y\leq  \beta^\veps \mu^\veps_{s^-}(\{2\veps\}) / (\veps \overline T) } } \, Q_2(\overline T ds,dy)  \\
		     & \ds + \int_0^t \int_{\Rb_+\times \Nb\backslash\{0,1\} } \veps (\delta_{\veps i+\veps}-\delta_{\veps i}) \indic{ \Enbrace{y\leq a^\veps_i u^\veps(s^-) \mu^\veps_{s^-}(\{\veps i\}) / (\veps^{2} \overline T) } } \, Q_3(\overline T ds,dy,di) \\
		     & \ds - \int_0^t \int_{\Rb_+ \times \Nb\backslash\{0,1,2\} } \veps (\delta_{\veps i}-\delta_{\veps i-\veps}) \indic{ \Enbrace{y\leq   b^\veps_i \mu^\veps_{s^-}(\{\veps i\}) /(\veps^{2} \overline T) } } \, Q_4(\overline T ds,dy,di)\,, 
 \end{array}
\end{equation}
with the balance law given, \as~for all $t\geq 0$, by
\begin{equation} \label{eq:rescaled_constraint}
 u^\veps(t) +  \lb \mu^\veps_t , \mathrm{Id} \rb  =  m^\veps\,,
\end{equation}
\subsection{The rescaled equation}

Let the characteristic function $\indic{S}$, for any set $S\in\Rb_+$, be $1$ on $S$ and $0$ elsewhere. We define the rate functions: 
\begin{equation} \label{eq:rescale_coeffs}
\begin{array}{rcll}
\ds  a^\veps(x)  &:=& \ds  a_i^\veps \indic{[\veps i,\veps (i+1))}, &\ds  \forall i \geq 2,   \\
\ds  b^\veps(x)  &:=& \ds  b_i^\veps \indic{[\veps i,\veps (i+1))}, &\ds  \forall i \geq 3. \\
\end{array}
\end{equation}

Again, the solution of (\ref{eq:rescaled_eds}-\ref{eq:rescaled_constraint}) is unique and has generator given by Equation \eqref{eq:gene_BD_stoch_rescale} in Definition \ref{def:rescaleSBD}. We note by $(\Fc_t^\veps)_{t\geq0}$ the canonical filtration associated to the process $\mu^{\veps}$.


In the following proposition, we test Equation \eqref{eq:rescaled_eds} against a function $\vphi$, measurable and locally bounded from $\Rb_+$ to $\Rb$. It will be useful to identify the limit equation. 

\begin{proposition}\label{cor:rescaled_process}
Let $\mu^\veps$ be the solution of (\ref{eq:rescaled_eds}-\ref{eq:rescaled_constraint}) for each $\veps>0$ and $\vphi$ a locally bounded and measurable real-valued function on $\Rb_+$. Then, for all $t\geq0$ 
\begin{equation}\label{eds_weak}
 \brak{\mu_t^\veps}{\vphi} = \brak{\mu_{\rm in}^\veps}{\vphi} + \Vc^{\veps,\vphi}_t + \Oc_t^{\veps,\vphi} \, , 
\end{equation}
where $\Vc^{\veps,\vphi}_t$ is the finite-variation part of $\brak{\mu_t^\veps}{\vphi}$ given by
\begin{multline}\label{eq:varfinie_rescale}
\ds \Vc^{\veps,\vphi}_t   =     \int_0^t \vphi(2\veps) \left[ \alpha^\veps u^\veps(s)(u^\veps(s)-\veps^2) - \beta^\veps \mu_{s}^\veps (\{2\veps\})  \right] \vphantom{\sum_{i\geq k}}\, ds \\
+   \int_0^t \int_{2 \veps}^{+\infty} \Delta_{\veps}(\vphi)  a^\veps(x) u^\veps(s) \mu_{s}^\veps(dx) \, ds \vphantom{\sum_{i\geq k}} \\
-  \int_0^t \int_{3\veps}^{+\infty} \Delta_{-\veps}(\vphi)(x)  b^\veps(x) \mu_{s}^\veps(dx)\, ds \vphantom{\sum_{i\geq k}} \,,
\end{multline}
with $\Delta_{h}(\vphi)(x) = (\vphi(x+h)-\vphi(x))/h$. Moreover,  $\Oc_t^{\veps,\vphi}$ is a $L^2-(\mathcal F_t^\veps)_{t\geq0}$ martingale starting from $0$ with (predictable) quadratic variation

\begin{multline}\label{eq:martingale_rescale}
 \langle \Oc^{\veps,\vphi} \rangle_t   =   \ds   \veps \int_0^t \vphi(2\veps)^2  \left[  \alpha^\veps u^\veps(s)(u^\veps(s)-\veps^2) + \beta^\veps \mu_{s}^\veps (\{2\veps\})  \right] \vphantom{\sum_{i\geq k}}\, ds \\
 +   \veps^{2} \int_0^t \int_{2\veps}^{+\infty}      \left( \Delta_{\veps}(\vphi)(x) \right)^2 a^\veps(x) u^\veps (s) \mu_{s}^\veps(dx) \, ds \vphantom{\sum_{i\geq k}} \\
 +  \veps^{2}\int_0^t \int_{3\veps}^{+\infty}     \left( \Delta_{-\veps}(\vphi)(x) \right)^2 b^\veps(x) \mu_{s}^\veps(dx)\, ds \vphantom{\sum_{i\geq k}} \,.
\end{multline}
\end{proposition}
We attempt to pass to the limit in \eqref{eds_weak} when enough compactness is available. We want the finite variation \eqref{eq:varfinie_rescale} to converge to the weak form of the Lifshitz-Slyozov operator (including boundary value) and the martingale \eqref{eq:martingale_rescale} to vanish (its quadratic variation) to recover a weak formulation of the deterministic problem at the limit in \eqref{eds_weak}. For that, we need moment estimates and tightness properties to obtain the compactness in the appropriate space. These are the results presented in the next section.

\section{Estimations and technical results} \label{sec:technical}

\subsection{Moment estimates}

To prove the convergence of $\{\mu^\veps\}$, the sequence of measure-valued SBD processes constructed as a solution of equations (\ref{eq:rescaled_eds}-\ref{eq:rescaled_constraint}), we will rely on compactness arguments (or tightness). These are achieved, in particular, thanks to moment estimates that are uniform with respect to $\veps$. In this section, we provide the appropriate estimates that will be necessary in the next sections. 

The first proposition provides a control of the total mass of the measure, namely $\brak{\mu_t^\veps}{\indic{}}$ where $\indic{}=\indic{\Rb_+}$, then a $L^\infty$ control of the free particle $u^\veps_t$ and of the first $x$-moment $\lb \mu^\veps_t , \Id \rb$.

\begin{proposition}\label{prop:moment}   
Let the hypotheses of Theorem \ref{thm:SBD_rescale_limit} hold. Let $\mu^\veps$ be the solution  of (\ref{eq:rescaled_eds}-\ref{eq:rescaled_constraint}) for each $\veps>0$. Then, for all $T>0$ we have that

\begin{align}
  &\sup_{\veps>0} \, \sup_{t\in[0,T]} u^\veps(t)  < +\infty \,, \ \as \label{eq:C_uni_bound} \\
  & \sup_{\veps>0} \, \sup_{t\in[0,T]} \lb \mu^\veps_t , \Id \rb < +\infty \,, \ \as \label{eq:mu,Id} \vphantom{ \esp{\sup_{t\in[0,T]}\brak{\mu_t^\veps}{\indic{}}} }\\
  & \sup_{\veps>0} \esp{\sup_{t\in[0,T]}\brak{\mu_t^\veps}{\indic{}}}  < + \infty\,. \label{eq:mu,1}
\end{align}
\end{proposition}

\begin{proof}


Since the sequence $\lbrace \mu^\veps_{\rm in}\rbrace$ lives \as\ in a compact $\Kc$ of $\Xc$, we have \as\ 

\[ \sup_{\veps>0} \, \lb \mu^\veps_{\rm in} , \indic{} \rb  < +\infty  \,, \text{ and } \sup_{\veps>0}\,  \lb \mu^\veps_{\rm in} , \Id \rb  < +\infty \,. \]
For the same reason, \as  
\[ \sup_{\veps>0} \, u^\veps_{\rm in}  < +\infty \,.\]
Thus, the conservation of mass \eqref{eq:rescaled_constraint} yields for all $t\in[0,T]$ and $\veps >0 $
\[u^\veps(t) \leq  m^\veps = u^\veps_{\rm in} + \lb \mu^\veps_{\rm in} , \Id \rb \,.\]
Thanks to our first remark, \eqref{eq:C_uni_bound} holds true by taking both supremum in time and $\veps$. We can similarly show \eqref{eq:mu,Id}. 

Let us now prove \eqref{eq:mu,1}. From the stochastic differential equation \eqref{eq:rescaled_eds} on $\mu^\veps$, dropping the non-positive terms, we have
\begin{equation} \label{estim_mu,1_without_negative_term}
 \brak{\mu_t^\veps}{\indic{}} \leq \brak{\mu_{\rm in}^\veps}{\indic{}} +  \int_0^t \int_{\Rb_+}  \veps \indic{ \Enbrace{y\leq  \alpha^\veps u^\veps(s^-)(u^\veps(s^-)-\veps^2)/(\veps \overline T) } } \, Q_1(\overline T ds,dy)\,.
\end{equation}
Thanks to \eqref{eq:C_uni_bound} and the convergence of $\alpha^\veps$ in \eqref{H1}, it exists a constant $K_0$ such that \as
\begin{equation} \label{eq:k_c_less_Kk}
\sup_{\veps>0} \, \sup_{t\in[0,T]} \alpha^\veps u^\veps(t)(u^\veps(t)-\veps^2) \leq K_0\,.
\end{equation}
Thus, in \eqref{estim_mu,1_without_negative_term}, taking the supremum in time on $[0,T]$
\begin{equation*}
 \sup_{t\in[0,T]} \brak{\mu_t^\veps}{\indic{}} \leq \brak{\mu_{\rm in}^\veps}{\indic{}} +  \int_0^T \int_{\Rb_+} \veps  \indic{ \Enbrace{y\leq K_0 /(\veps \overline T)} } \, Q_1(\overline T ds,dy). 
\end{equation*}
Thanks to the uniform \as~bound on the initial moment, represented here by the constant $K_1$, we conclude by taking the mean that
\begin{equation*}
 \esp{ \sup_{t\in[0,T]} \brak{\mu_t^\veps}{\indic{}} } \leq K_1 +  K_0 T\,.
\end{equation*} 

\end{proof}

The method we use here to prove the convergence needs a uniform control on superlinear moments of $\mu^\veps$ so that it controls the tail at infinity. To that, let us introduce the set $\Uc_1$ of nonnegative functions $\Phi$, convex and belonging to $\Cc^1([0,+\infty))\cap W^{2,\infty}_{loc}((0,+\infty))$ such that $\Phi(0)=0$, $\Phi'$ is concave and $\Phi'(0)\geq 0$ and the set $\Uc_\infty$ of nonnegative increasing convex functions $\Phi$ such that
\[ \lim_{x\to +\infty} \frac{\Phi(x)}{x} = + \infty\,. \]
We denote by $\Uc_{1,\infty} := \Uc_{1} \cap \Uc_{\infty}$. These functions have remarkable properties when conjugate to the structure of the SBD equation and provide important estimates as in the deterministic case, see for instance \cite{Laurenccot2002}. 

\begin{remark}
Any function $x\mapsto x^{1+\eta}$ with $\eta \in(0,1)$ belongs to the set $\Uc_{1,\infty}$. Such a function $\Phi$ in $\Uc_{1,\infty}$ roughly represents a moment slightly greater than $1$.  
\end{remark}
The next proposition provides the propagation of an extra $x$-moment of the rescaled measure-valued SBD process that will be necessary to prove the tightness in $(\Xc,w)$.

\begin{proposition} \label{prop:moment_phi}
Let the hypotheses of Theorem \ref{thm:SBD_rescale_limit} hold. Let $\mu^\veps$ be the solution  of (\ref{eq:rescaled_eds}-\ref{eq:rescaled_constraint}) for each $\veps>0$. It exists $\Phi_1 \in \Uc_{1,\infty} $ such that, for all $T>0$ we have
 
 \[ \sup_{\veps>0} \esp{\sup_{t\in[0,T]} \lb \mu_t^\veps ,\Phi_1 \rb} < +\infty \,.\]
 
\end{proposition}
 
\begin{proof}
As in the proof of Proposition \ref{prop:moment}, we know that there exists a compact $\Kc$ of $(\Xc,w)$ such that \as~for all $\veps>0$, $\mu^\veps_{\rm in}\in\Kc$. Moreover, for this compact $\Kc$, thanks to refined version of the De La Vall\'ee-Poussin lemma \cite{chauhoan1977,Hingant2015}, it exists $\Phi_1\in\Uc_{1,\infty}$ such that 
\[\sup_{\nu\in \Kc} \, \lb \nu,\Phi_1\rb < +\infty \,.\]
Thus, we obtain

\[\sup_{\veps>0} \, \lb\mu^\veps_{\rm in},\Phi_1\rb < +\infty \,, \ \as \]

Now, let $T>0$. As $\Phi_1$ is nonnegative and increasing (because $\Phi_1(x)\leq x \Phi_1'(x)$ by convexity of $\Phi_1$ and $\Phi_1(0)=0$, so that $\Phi_1'$ is non-negative), we may also drop the non-positive terms to obtain
 
 \begin{multline*}
 \ds \lb \mu_t^\veps ,\Phi_1 \rb   \leq   \ds \lb \mu_{\rm in}^\veps ,\Phi_1 \rb   +  \int_0^t \int_{\Rb_+} \veps \Phi_1(2\veps) \indic{  \Enbrace{y\leq  \alpha^\veps u^\veps(s^-)(u^\veps(s^-)-\veps^2)/(\veps \overline T) } } \, Q_1(\overline T ds,dy)  \\
 + \int_0^t \int_{\Rb_+\times  \Nb\backslash\{0,1\} } \veps \left( \Phi_1(\veps i +\veps)-\Phi_1(\veps i)\right)   \indic{ \Enbrace{ y\leq    a^\veps_i u^\veps(s^-) \mu^\veps_{s^-}(\{\veps i\}) / (\veps^{2} \overline T) }}   Q_3(\overline T ds,dy,di) \,.
\end{multline*}
Then, using the convexity of $\Phi_1$, the concavity of $\Phi_1'$ and then its non-increasing right derivative (denoted $\Phi_{1,\,r}''$),  we have, for all $i\geq 2$
\begin{equation*}
  \Phi_1(\veps i +\veps)-\Phi_1(\veps i) \leq \veps \Phi_1'(\veps i +\veps)  \leq \veps \left( \Phi_1'(\veps i) + \veps \Phi_{1,\, r}''(0) \right)\,.
\end{equation*}
Now, using the bound \eqref{eq:C_uni_bound} denoted by $K_1$, the bound \eqref{eq:k_c_less_Kk}, taking the supremum in time, and then the expectation, it entails
\begin{multline} \label{esp_convex}
\esp{\sup_{\sigma \in(0,t)} \lb \mu_\sigma^\veps ,\Phi_1 \rb }  \leq    \esp{ \lb \mu_{\rm in}^\veps ,\Phi_1 \rb}     +   \veps  \Phi_1(2\veps) K_0 T  \\
 +  K_1 \int_0^t  \esp{ \sup_{ \sigma \in(0,s)} \int_{\Rb_+} (\Phi_1'(x)+\veps \Phi_{1,\, r}''(0))    a^\veps(x)  \mu^\veps_{\sigma}(dx)  } ds  \,.
\end{multline}
Since  $\Phi_1\in \Uc_{1,\infty}$ we have $x\Phi_1'(x) \leq 2 \Phi_1(x)$ for all $x\geq 0$ by \cite[Lemma A.1]{Laurencot2001}. Thus for any $R>0$ and thanks to hypothesis \eqref{H3} on $a^\veps$ we obtain
\begin{multline*} 
 \int_{0}^{+\infty} \Phi_1'(x)    a^\veps(x)  \mu^\veps_{t}(dx)  = \int_{0}^{R} \Phi_1'(x)    a^\veps(x)  \mu^\veps_{t}(dx) + \int_{R}^{+\infty} \Phi_1'(x)    a^\veps(x)  \mu^\veps_{t}(dx) \\
 \leq K_a \int_{0}^{R}  (1+x) \Phi_1'(x)  \mu^\veps_{t}(dx)  + K_a\left(\frac 1 R + 1\right) \int_{R}^{+\infty} x\Phi_1'(x)  \mu^\veps_{t}(dx)\\
  \leq K_a \left(\sup_{x\in(0,R)} \Phi_1'(x) \right) \lb \mu^\veps_{t},\indic{}\rb +  2 K_a \lb \mu^\veps_{t},\Phi_1 \rb + \frac{ 2 K_a}{ R}  \lb \mu^\veps_{t},\Phi_1 \rb \,.\vphantom{\int}
\end{multline*}
Thus, there exists a constant $K$ which depends on $\Phi_1$, $R$, the uniform bound \eqref{eq:mu,1} on $\lb \mu_t^\veps,\indic{} \rb$ and on $K_a$ such that 
\begin{equation}  \label{esp_convex_2}
 \esp{\sup_{ \sigma \in(0,s)} \int_{0}^{+\infty} \Phi_1'(x)    a^\veps(x)  \mu^\veps_{\sigma}(dx)  } \leq K \left( 1+ \esp{ \sup_{ \sigma \in(0,s)} \lb \mu^\veps_{\sigma},\Phi_1 \rb } \right) \vphantom{\int} \, .
\end{equation}
Also, there exists a constant,  still denoted by $K$, which depends on the uniform bound \eqref{eq:mu,1} on $\lb \mu_t^\veps,\indic{} \rb$, the uniform bound \eqref{eq:mu,Id} on  $\lb \mu_t^\veps,\mathrm{Id} \rb$  and on $K_a$  such that
\begin{equation}  \label{esp_convex_3}
 \esp{\sup_{ \sigma \in(0,s)} \int_{0}^{+\infty}  a^\veps(x)  \mu^\veps_{\sigma}(dx)  } \leq K \, .
\end{equation}
Finally, combining \eqref{esp_convex}, \eqref{esp_convex_2} and \eqref{esp_convex_3}, there exists a constant, still denoted by $K$, such that for all $t\in[0,T]$

\begin{equation*} 
\esp{\sup_{\sigma \in(0,t)} \lb \mu_\sigma^\veps ,\Phi_1 \rb }  \leq    K \left( 
 1 +   \int_0^t  \esp{ \sup_{ \sigma \in(0,s)} \lb \mu^\veps_{\sigma},\Phi_1 \rb) }  ds \right)  \, .
\end{equation*}
We get the desired estimation using the Gronwall lemma. 
\end{proof}

We will also need a superlinear control on the total mass $\brak{\mu_t^\veps}{\indic{}}$ avoiding concentration. It will notably be useful to treat the boundary condition.

\begin{proposition}\label{prop_mu_1}
Let the hypotheses of Theorem \ref{thm:SBD_rescale_limit} hold. Let $\mu^\veps$ be the solution  of (\ref{eq:rescaled_eds}-\ref{eq:rescaled_constraint}) for each $\veps>0$. It exists $\Phi_2 \in \Uc_{1,\infty} $ such that, for all $T>0$ we have
\begin{equation*}
  \sup_{\veps>0} \esp{\sup_{t\in[0,T]} \Phi_2( \brak{\mu_t^\veps}{\indic{}})}  < + \infty\,.
\end{equation*}
\end{proposition}

\begin{proof}
Since $\mu_{\rm in}^\veps$ converges in law towards the deterministic measure $\mu_{\rm in}$, we deduce that $\brak{\mu_{\rm in}^\veps}{\indic{}}$ converges also in law to the deterministic value $\brak{\mu_{\rm in}}{\indic{}}$. Therefore, a subsequence of $\brak{\mu_{\rm in}^\veps}{\indic{}}$ converges in $L^1$ towards $\brak{\mu_{\rm in}}{\indic{}}$. Again, thanks to refined version of the de La Vall\'ee-Poussin lemma \cite{chauhoan1977,Hingant2015}, it exists $\Phi_2\in\Uc_{1,\infty}$ such that 
\[\sup_{\veps>0} \, \esp{ \Phi_2(\lb \mu^\veps_{\rm in},\indic{}\rb) } < +\infty \,.\]

Now, from the stochastic differential equation \eqref{eq:rescaled_eds} and using that $\Phi_2$ is increasing we drop the non-negative terms
\begin{multline*}
   \Phi_2( \lb \mu_t^\veps , \indic{} \rb)  \leq  \Phi_2( \lb \mu_{\rm in}^\veps , \indic{} \rb) \\ +  \int_0^t \int_{\Rb_+} \Phi_2( \lb \mu_{s^-}^\veps , \indic{} \rb + \veps)  - \Phi_2( \lb \mu_{s^-}^\veps , \indic{} \rb) ] \indic{ \Enbrace{y\leq  \alpha^\veps u^\veps(s^-)(u^\veps(s^-)-\veps^2)/ (\veps \overline T)} } \, Q_1(\overline T ds,dy)  \, .
\end{multline*}
Then the proof follows by the same arguments as Proposition \ref{prop:moment_phi}.
\end{proof}

\subsection{Tightness of the rescaled process}

The aim of this section is to prove the following tightness property of the family $\{\mu^\veps\}$ of $\Xc$-valued processes.  

\begin{proposition}\label{prop:tight_weak}
Let the hypotheses of Theorem \ref{thm:SBD_rescale_limit} hold. Let $\mu^\veps$ be the solution  of (\ref{eq:rescaled_eds}-\ref{eq:rescaled_constraint}) for each $\veps>0$. Then $\{\mu^\veps\}$ is tight in $\Pc(\Dc(\Rb_+,w-\Xc))$ and $\{u^\veps\}$ is tight in $\Pc(\Dc(\Rb_+,\Rb_+))$. Moreover, any accumulation point $\mu$ of $\{\mu^\veps\}$ belongs \as~to $\Cc(\Rb_+,w-\Xc)$.
\end{proposition}
The proof of this result rests on the Aldous criterion for tightness \cite[p 176]{Billingsley99}. It is in two parts, first the compact containment condition and then the equicontinuity (in the sense of c\`adl\`ag). For the former we need to make explicit a weakly compact of $\Xc$.

\begin{lemma}\label{lem:compactcontain}
 Under the same assumptions as Proposition \ref{prop:tight_weak}, for all $T>0$ and $\eta$ sufficiently small, there exists a compact $\Kc_{\eta,T}$ of $(\Xc,w)$ such that
 
 \[ \Pb\left( \left\{ \mu^\veps_t \in \Kc_{\eta,T} \, : \, 0\leq t  \leq T \right\}\right) \geq 1 -\eta \,. \]
\end{lemma}

\begin{proof}
Let $\eta \in (0,1)$ and $T>0$. Thanks to Propositions \ref{prop:moment} and \ref{prop:moment_phi}, we define three constants 
 
 \begin{equation*}
  C^1_{\eta,T}= 3 \sup_\veps \esp{\sup_{t\in(0,T)} \lb \mu_t^\veps ,\Phi_1 \rb}/ \eta, \ 
  C^2_{\eta,T}= 3 \sup_\veps \esp{\sup_{t\in(0,T)} \lb \mu_t^\veps , \Id \rb}/ \eta,
 \end{equation*}
 and,
 \begin{equation*}
   C^3_{\eta,T}= 3 \sup_\veps \esp{\sup_{t\in(0,T)} \lb \mu_t^\veps ,\indic{} \rb}/ \eta. \ 
 \end{equation*}
 We then introduce the weakly relatively compact set of  $\Xc$, by Lemma \ref{lem:compactX} in Appendix,
 \[ \Kc_{\eta,T} = \Big\{ \nu \in \Xc :  \lb \nu ,\Phi_1 \rb \leq C^1_{\eta,T}, \ \lb \nu ,\Id \rb \leq C^2_{\eta,T},\  \lb \nu ,\indic{} \rb \leq C^3_{\eta,T} \Big\}\,.   \]
 We have, by Markov's inequality, that
 \begin{equation*} \Pb\left( \left\{\sup_{t\in(0,T)} \lb \mu^\veps_t , \Phi_1 \rb \geq C^1_{\eta,T} \right\} \cup \left\{\sup_{t\in(0,T)} \lb \mu^\veps_t , \Id \rb \geq C^2_{\eta,T} \right\} \cup \left\{\sup_{t\in(0,T)} \lb \mu^\veps_t , \indic{} \rb \geq C^3_{\eta,T} \right\}\right) \leq \eta \,,
 \end{equation*}
 providing the desired compact containment condition. 
 \end{proof}
 The second step of the proof gives the equicontinuity property of the process on the c\`adl\`ag space. To that we introduce a metric $d_\Xc$ equivalent to the weak convergence on $\Xc$. We follow \cite{Fournier2004} to define a sequence $\{\vphi_k\}$ of functions in  $\Cc_b^1(\Rb_+)$  such that $\lVert \vphi_k \rVert_\infty +  \lVert \vphi_k' \rVert_\infty \leq 1$ and for all $(\nu_1,\nu_2)\in \Xc\times \Xc$,
 \begin{equation} \label{alternate_metric_on_X}
 d_\Xc(\nu_1,\nu_2) = \sum_{k\geq 1} 2^{-k} \lvert \lb (\indic{}+\Id ) \cdot \nu_1 , \vphi_k \rb - \lb (\indic{}+\Id ) \cdot \nu_2 , \vphi_k \rb \rvert\,. 
 \end{equation}
 
\begin{lemma}\label{lem:equic_weak}
 Under the same assumptions as Proposition \ref{prop:tight_weak}, for all $\eta>0$, there exists $h>0$ such that 
  
  \[ \sup_ \veps \sup_t \sup_{s \in (0,h) } \esp{  d_\Xc ( \mu^\veps_{t+s} ,  \mu^\veps_{t} ) } \leq \eta \,. \]

\end{lemma}

\begin{proof}
Let $\{\vphi_k\}$ be as described above. For any $k\in \Nb$, we define $\psi_k(x)=(1+x) \vphi_k(x)$ and then the approximation $\psi_{k,R} (x) = (1+x) \vphi_k(x)$ if $x<R$ and $\psi_{k,R} (x) = (1+R) \vphi_k(x)$ otherwise. Then, we get for all $h <T$ and $t\in [0,T-h]$ with  $s \in(0,h)$, 
\begin{multline} \label{eq:estim_0}
 \lvert \lb (\indic{}+\Id ) \cdot \mu^\veps_{t+s},\vphi_k \rb - \lb (\indic{}+\Id ) \cdot \mu^\veps_t,\vphi_k \rb \rvert \\
 \leq  \lvert \lb \mu^\veps_{t+s},\psi_{k,R} \rb - \lb \mu^\veps_t,\psi_{k,R} \rb \rvert 
+ 2 \sup_{t\in[0,T]} \lvert \lb  \mu^\veps_t, (\Id -R) \vphi_k \indic{[R,+\infty)} \rb \rvert\, .
\end{multline}
For the second term in Equation \eqref{eq:estim_0}, since $\lVert \vphi_k \rVert_\infty +  \lVert \vphi_k' \rVert_\infty \leq 1$, we obtain
\begin{equation*}
 \sup_{\veps} \esp{ \sup_{t\in[0,T]} \lvert \lb  \mu^\veps_t,  (\Id -R) \vphi_k \indic{[R,+\infty)}  \rb \rvert }
 \leq \left(\sup_{x>R} \frac{x}{\Phi_1(x)} \right)   \sup_{\veps} \esp{ \sup_{t\in[0,T]}  \lb \mu^\veps_t, \Phi_1 \rb }.
 \end{equation*}
By Proposition \ref{prop:moment_phi} and since $\Phi_1$ belongs to $\Uc_\infty$, for $\eta>0$ it exists $R$ large enough such that 
\begin{equation} \label{eq:estim_1}
  \sup_{\veps} \esp{ \sup_{t\in[0,T]} \lvert \lb  \mu^\veps_t, (\Id -R) \vphi_k \indic{[R,+\infty)}  \rb \rvert } \leq \frac \eta 4 \, .
\end{equation}
For the first term in Equation \eqref{eq:estim_0}, we have with the notations of Corollary \ref{cor:rescaled_process},
\begin{equation*}
 \lvert \lb \mu^\veps_{t+s},\psi_{k,R} \rb - \lb \mu^\veps_t,\psi_{k,R} \rb \rvert 
  \leq \lvert \Vc^{\veps,\psi_{k,R}}_{t+s}-\Vc^{\veps,\psi_{k,R}}_t  \rvert  + \sup_{s\in(0,h)} \lvert \Oc^{\veps,\psi_{k,R}}_{t+s}-\Oc^{\veps,\psi_{k,R}}_t  \rvert \,.
\end{equation*}

Then, taking expectation, applying the Cauchy-Schwarz inequality in the martingale term and then the Burkholder-Davis-Gundy inequality \cite[Theorem 48, p. 193]{Protter}, both on the martingale term, we get
\begin{equation*}
 \sup_{s \in (0,h) } \esp{ \lvert \lb \mu^\veps_{t+s},\psi_{k,R} \rb - \lb \mu^\veps_t,\psi_{k,R} \rb \rvert }
 \leq \sup_{s \in (0,h) } \esp{ \left\lvert \Vc^{\veps,\psi_{k,R}}_{t+s}-\Vc^{\veps,\psi_{k,R}}_t  \right\rvert } + \esp{ \left| \lb \Oc^{\veps,\psi_{k,R}} \rb_{t+h} - \lb \Oc^{\veps,\psi_{k,R}} \rb_t \right| } \,.
\end{equation*}
The two terms above can now be treated thanks to moment estimates and Assumption \ref{hyp:cv_fct}. By the convergence of $\beta^\veps$ and the moment estimate \eqref{eq:mu,1}, it exists a positive constant still denoted by $K_0$ such that 
\begin{equation}\label{eq:b2_mu,2_bound}
\sup_{\veps>0} \esp{\sup_{t\in[0,T]} \beta^\veps  \mu^\veps_t(\{2\veps\})} \leq K_0 \,.
\end{equation}

Hence, since $\lVert \vphi_k \rVert_\infty +  \lVert \vphi_k' \rVert_\infty \leq 1$, by \eqref{eq:k_c_less_Kk} and \eqref{eq:b2_mu,2_bound}, and by the hypotheses (\ref{H3}-\ref{H4}) on $a$ and $b$, we have for all $s\in(0,h)$
\begin{equation*} 
  \Abs{ \Vc^{\veps,\psi_{k,R}}_{t+s}-\Vc^{\veps,\psi_{k,R}}_t } \leq 2K_0  (1+R) h 
 + \left( K_a \sup_{t\in[0,T]} u^\veps_t + K_b   \right)  \Big{[} \sup_{t\in[0,T]}\brak{\mu_t^\veps}{\indic{}}+\sup_{t\in[0,T]}\brak{\mu_t^\veps}{\mathrm{Id}}  \Big{]} (2+R) h \,,
\end{equation*}
and,
\begin{multline*} 
 \Abs{ \lb \Oc^{\veps,\psi_{k,R}} \rb_{t+h}- \lb \Oc^{\veps,\psi_{k,R}}\rb_t } \leq \veps 2 K_0 (1+R)^2 h \\
 + \veps^{2}\left(K_a \sup_{t\in[0,T]} u^\veps_t + K_b \right)  \Big{[} \sup_{[0,T]}\brak{\mu_t^\veps}{\indic{}}+\sup_{[0,T]}\brak{\mu_t^\veps}{\mathrm{Id}}  \Big{]} (2+R)^2 h \,.
\end{multline*}
Thus, using estimates in Proposition \ref{prop:moment}, we deduce that it exists $h$ (small enough) such that: 

\begin{equation} \label{eq:estim_2}
 \sup_{s \in (0,h) } \esp{ \lvert \lb \mu^\veps_{t+s},\psi_{k,R} \rb - \lb \mu^\veps_t,\psi_{k,R} \rb \rvert } \leq \frac \eta 2 \,,
\end{equation}
where the estimation is uniform in $\veps$ and $t$. We conclude by combining \eqref{eq:estim_1} and \eqref{eq:estim_2} into \eqref{eq:estim_0}.

 \end{proof}
Finally, we finish by the continuity of the limit process.

\begin{lemma}\label{lem:continu}
  Let the same assumptions as Proposition \ref{prop:tight_weak} hold. Let $\mu$ be an accumulation point of $\{\mu^\veps\}$ in $\Pc(\Dc(\Rb_+,w-\Xc))$. Then,  $\mu$ \as~belongs to $\Cc(\Rb_+,w-\Xc)$.
\end{lemma}
\begin{proof}
First, note that by construction since $\mu^\veps$ takes values in $\Mc_\delta^\veps$ and $\sum_{i\geq 1} \veps i\mu^\veps_t(\{\veps i\}) = \lb \mu^\veps_t ,\Id \rb \leq m^\veps$, thus for for each $\veps>0$, $\mu^\veps$ is compactly supported in $[2 \veps,m^\veps/\veps]$.
Hence, using the stochastic differential equation \eqref{eq:rescaled_eds}, we obtain for all $\vphi$  with $\lVert \vphi \rVert_\infty +  \lVert \vphi' \rVert_\infty \leq 1$ that 
\begin{multline*}
 \lvert \brak{(\indic{}+\Id)\cdot \mu_s^\veps}{\vphi} - \brak{(\indic{}+\Id)\cdot \mu_{s^-}^\veps}{\vphi} \rvert \\
 \leq \veps \Big{(} \lvert (1+2\veps) \vphi(2\veps)\rvert + \sup_{x \in (2 \veps,m^\veps/\veps)} \lvert (1+x+\veps) \vphi(x+\veps)- (1+x)\vphi(x) \rvert \Big{)}\\
 \leq  \veps  (1+2\veps) +  \veps^{2} \left( \frac{K_2}{\veps} + 2\right)\,,
\end{multline*}
where $K_2$ is a constant such that \as~$m^\veps<K_2$, by Proposition \ref{prop:moment}.
We deduce that for all $T\geq 0$, we have \as~ $\lim_{\veps\to0} \sup_{s\in[0,T]}  d_\Xc( \mu_s^\veps, \mu_{s^-}^\veps) = 0$. This concludes the proof.

 \end{proof}

\begin{proof}[Proposition \ref{prop:tight_weak}] The tighness of $\{\mu^\veps\}$ readily follows from Lemma \ref{lem:compactcontain} and Lemma \ref{lem:equic_weak} which are the Aldous criterion of tightness given in \cite[p 176]{Billingsley99}. The continuity is a direct consequence of Lemma \ref{lem:continu}. The tightness of $\{u^\veps\}$ follows immediately from its definition and the properties obtained on $\{\mu^\veps\}$. 
\end{proof}

\subsection{Tightness of the boundary term} 

While trying to pass to the limit in Equation \eqref{eq:varfinie_rescale} in $(\Xc,w)$, we have to deal with the term $\mu_t^\veps(\{2\veps\})$, for which we need to prove also a tightness property. However, when looking at the time evolution equation of $\mu_t^\veps(\{2\veps\})$, it appears that such a term may evolve at a faster time scale than $\mu_\veps$, viewed as a measure in $\Xc$.  We use ideas from \cite{Kurtz1992}, and separate the action of $\mu^\veps$ as a measure on $\Rb_+$ (for large size) and the evaluation at small sizes (on $\veps\Nb$). But, the equation on $\mu_t^\veps(\{2\veps\})$ involves $\mu_t^\veps(\{3\veps\})$, the latter involves $\mu_t^\veps(\{4\veps\})$, {\it etc}. Thus, we need to consider together all the evaluations of the measure at points $i\veps$. That is, for all $\veps$ and all $t$, we define as in Section \ref{sec:main}, the sequence $p^{\veps}_t=(p^{\,\veps}_{n,t})_{n \in \Nb} \in \ell^1_+$ by
\begin{equation} \label{def:p}
p_{n,t}=\mu_t^\veps(\{\veps(n+2)\}) \,, \forall n \geq 0\,.
\end{equation}
%
We recall that, for each $\veps$, this sequence is by definition compactly supported between $0$ and $m^{\veps}/\veps^{2} -2$ since $\mu^\veps$ is supported in $[2\veps,m^\veps/\veps]$. For the remainder we let in $\ell_1^+$ the canonical sequences $(\indic{k})_{k \in \Nb}$ be defined for all $k$ by $\indic{k,n}=0$ for all $n\neq k$ and $\indic{k,k}=1$.

\noindent The following proposition is immediate, but makes clear the difference of time scales.

\begin{proposition}\label{prop:generateur_p}
 Let $p^\veps$ given by \eqref{def:p} for each $\veps>0$. Then, $p^\veps$ is  a $\ell_1^+$-valued c\`adl\`ag process. Its infinitesimal generator is defined, for all measurable and locally bounded $g$ from $\ell^+_1$ to $\Rb$ and compactly supported $q\in \ell^+_1$, by
\begin{multline}\label{eq:Hveps}
\Hc^{\veps} g (q)= \ds \frac{g(q+ \veps \indic{0}) - g(q)}{\veps} \alpha^\veps c(c-\veps^2) \vphantom{\sum_{i\geq k}}
 + \ds \frac{g(q- \veps \indic{0}) - g( q)}{\veps} \beta^\veps q_0 \vphantom{\sum_{i\geq k}} \\
  + \veps^{-(1-r_a)}   \sum_{n\geq0}  \frac{g(q+ \veps (\indic{n+1}-\indic{n})) - g( q)}{\veps}  \frac{a^{\veps}(\veps(n+2))}{\veps^{r_a}}   c q_n\\
 + \veps^{-(1-r_b)}  \sum_{n \geq 1} \frac{g(q- \veps (\indic{n}-\indic{n-1})) - g( q)}{\veps} \frac{b^{\veps}(\veps(n+2))}{\veps^{r_b}} q_n \vphantom{\sum_{i\geq k}}\,,
\end{multline}
where $c=m^\veps - \veps \sum_{n\geq0} (n+2) q_n$. Moreover, for such a function $g$,
\begin{equation*}
g(p^\veps_t)-g(p^\veps_{\rm in})-\int_0^t \Hc^\veps g(p^\veps_s)\, ds
\end{equation*}
is a $L^1-(\Fc_t^\veps)_{t\geq0}$  martingale.
\end{proposition}

\begin{remark} \label{topo_l1}
 We recall the topology on $\ell_1^+$ is the vague topology $v-\ell^+_1$, or alternatively $(\ell^+_1,v)$, \ie~the topology of the convergence,   $q^\veps \to q$ in $v-\ell^+_1$ if and only if $\sum_{n\in \Nb} q_n^\veps \vphi_n \to \sum_{n\in \Nb} q_n \vphi_n$ for all $\vphi\in \ell_0$ (the sequences vanishing at infinity). Remark, the space is not the Banach space, as usual, but is a Polish space (consider for instance its canonical homeomorphism with $(\Mc_b(\Nb),v)$).
\end{remark}

\begin{remark}
The scaling exponents $r_a$ and $r_b$ in Equation \eqref{eq:Hveps} are those given by Assumption \ref{hyp:coef_BD} and ensure that both $a^{\veps}(\veps(2+n))/\veps^{r_a}$ and $b^{\veps}(\veps(2+n))/\veps^{r_b}$ stay bounded and converge to positive values as $\veps \to 0$. Hence the exponents $-(1-r_a)$ and $-(1-r_b)$ are really proper time scales of the infinitesimal generator $\Hc^{\veps}$. This implies that the main part of the generator depends on the values of $r_a$ and $r_b$. In particular, since $\min(r_a,r_b)<1$, the time scale of the infinitesimal generator $\Hc^{\veps}$ is faster than the time scale of the generator $\mathcal L^\veps$ of $\mu^\veps$.
\end{remark}
Using ideas from \cite{Kurtz1992}, we want to prove a tightness result for the sequence $\{p^\veps\}$ in a space that \emph{do not see} the fast variations of $p^\veps$. For this we use the occupation measure  $\Gamma^\veps$ defined by \eqref{def:Gamma}. The following proposition states the relative compactness of $\{\Gamma^\veps\}$. 

\begin{proposition}\label{prop:tightbord_measure}
 Let $\Gamma^\veps$ be defined by \eqref{def:Gamma} for each $\veps>0$ and assume that (only) Assumption \ref{hyp:cv_fct} holds.  Then  $\{\Gamma^\veps\}$ is tight in $\Pc(w^\#-\Yc)$. 
\end{proposition}

\begin{proof}
 We start by proving that for all $t\geq  0$, the restriction of $\Gamma^{\veps}$ to $[0,t]\times \ell_1^+$ belongs (uniformly in $\veps$) to a compact of 
 
 \[ \Yc_t := \left\{ \Theta \in \Mc_b([0,t]\times \ell_1^+) \, : \, \forall t\geq 0,\ \int_{[0,t]\times \ell_1^+} \lVert q \rVert_{\ell_1} \, \Theta(ds\times dq) < +\infty  \right\} \, ,\]
 for the weak topology defined in \ref{app:space_X}. Indeed, by definition we have $\lVert p^\veps_t \rVert_{\ell_1}=\lb \mu_t^{\veps}, \indic{} \rb$, and we get
 \begin{equation*}
  \int_{[0,t]\times \ell_1^+} (1+\lVert q \rVert_{\ell_1} + \Phi_2(\lVert q \rVert_{\ell_1}) )\Gamma^\veps(ds\times dq)
  \leq t \left( 1+ \sup_{s\in[0,t]} \lb \mu^\veps_s,\indic{}\rb + \sup_{s\in[0,t]} \Phi_2(\lb \mu^\veps_s,\indic{}\rb ) \right)\,.
 \end{equation*}
 Thanks to Propositions \ref{prop:moment} and \ref{prop_mu_1}, we easily check that
 \[ \sup_{\veps>0}\esp{ \int_{[0,t]\times \ell_1^+} (1+\lVert q \rVert_{\ell_1} + \Phi_2(\lVert q \rVert_{\ell_1})) \Gamma^\veps(ds\times dq)} <+\infty\,.\]
 The latter yields, by Lemma \ref{lem:compactX} (in Appendix) and remarking that the bounded subsets of $(\ell_1^+,v)$ are relatively compact, that for all $t\geq 0$ the sequence  $\{\Gamma^{\veps\,, t}\}$ belongs to a weak compact set $\Kc_t$ of $(\Yc_t,w^\#)$.
 
\noindent Now, we let  a sequence $\{t_k\}$ such that $\lim_{k\to \infty} t_k = +\infty$ and we let $\eta > 0$. We construct a sequence $\{C_{k,\eta}\}$ of positive constants such that 
 \[ \sup_{\veps>0}\esp{ \int_{[0,t_k]\times \ell_1^+} (1+\lVert q \rVert_{\ell_1} + \Phi_2(\lVert q \rVert_{\ell_1})) \Gamma^\veps(ds\times dq)} \leq  \eta\,  C_{k,\eta} 2^{-k}\,.\]
 We then define the weak compact set $\Kc_{t_k,\eta}$ of  $(\Yc_t,w^\#)$ consisting of measures $\Theta\in \Mc_b([0,t_k]\times\ell_1^+) $ such that 
 \[  \int_{[0,t_k]\times \ell_1^+} (1+\lVert q \rVert_{\ell_1} + \Phi_2(\lVert q \rVert_{\ell_1})) \Theta(ds\times dq) \leq C_{k,\eta} \,.\]
 It follows by Markov inequality that
 \begin{equation*}
  \prob{\Gamma^{\veps\,, t_k} \in \Kc^c_{t_k,\eta}}
  \leq  \frac{\sup_{\veps>0}\esp{ \int_{[0,t_k]\times \ell_1^+} (1+\lVert q \rVert_{\ell_1} + \Phi_2(\lVert q \rVert_{\ell_1})) \Gamma^\veps(ds\times dq)}}{C_{k,\eta}}
  \leq \eta 2^{-k}\,.
 \end{equation*}
 Thus, letting $\Kc_\eta = \{\Theta \in \Yc(\Rb_+\times\ell_1^+)\, :\, \forall k\geq 0,\ \Theta^{t_k}\in\Kc_{t_k,\eta}\}$, we obtain that 
 \[\prob{\Gamma^\veps \in \Kc_\eta^c} \leq \sum_{k\geq 0} \prob{\Gamma^{\veps\,,t_k}\in \Kc_{t_k,\eta}^c} \leq \eta\,.\]
 As $\Kc_\eta$ defines a compact of $(\Yc,w^\#)$ for any $\eta >0$ by Lemma \ref{lem:compactY}, this proves the tightness of $\{\Gamma^\veps\}$. 
 \end{proof}

\section{Convergence and limit problem} \label{sec:proof}


At this point, we already know, thanks to Propositions \ref{prop:tight_weak} and \ref{prop:tightbord_measure}, that $\{\mu^{\veps_n}\}$ and $\{\Gamma^{\veps_n}\}$ are tight respectively in $\Pc(\Dc(\Rb_+,w-\Xc))$ and $\Pc(w^\#-\Yc)$. Thus, the couple $\{(\mu^\veps,\Gamma^\veps)\}$ is tight in $\Pc\big(\Dc(\Rb_+,w-\Xc)\times (\Yc,w^\#)\big)$ and converges in law, up to a subsequence still indexed by $\veps$, to an accumulation point $\{(\mu,\Gamma)\}$. The topology on $\Dc(\Rb_+,w-\Xc) \times (\Yc,w^\#)$ is the product topology that remains a Polish space, and the convergence in law is the classical one for probability measure.  

The aim of this section is to identify the limit problem, \ie~ to recover the main results stated at the beginning. We start by proving Theorems \ref{thm:SBD_rescale_limit} and  \ref{thm:LS_convvague}.

\subsection{Identification of the limit - Proof of Theorems \ref{thm:SBD_rescale_limit} and \ref{thm:LS_convvague}} 

This section is first devoted to the proof of Theorem \ref{thm:SBD_rescale_limit}, which identifies the equation satisfied by the limit $\mu$.  As we see, it is a Lifshitz-Slyozov equation (in the weak sense) with boundary terms depending on integrals against $\Gamma$ (``averages''). Theorem \ref{thm:LS_convvague} is in fact a particular case, and the proof will directly follow as mentioned in Section \ref{sec:main}.

The proof of Theorem \ref{thm:SBD_rescale_limit} relies on the identification of the limit through a functional, that stands for the limit model, studied along the process. Thus, for any given $\vphi \in \mathcal C_b^1(\Rb_+)$ and $t\geq 0$, we define for all $(\nu,\Theta) \in  \Dc(\Rb_+,w-\Xc) \times \Yc$, the functional
\begin{equation} \label{functional_F}
 F_t^\vphi(\nu,\Theta) = \brak{\nu_t}{\vphi}-\brak{\nu_{\rm in}}{\vphi}-D^\vphi_t (\nu) - B^\vphi_t (\nu) - \tilde B^\vphi_t (\Theta) \,,
\end{equation}
where $D^\vphi_t$ denotes the drift, $B^\vphi_t$ and $\tilde B^\vphi_t$ the boundary terms, respectively given by
\begin{equation*}
 \begin{array}{rcl}
  \ds D^\vphi_t(\nu) & = & \ds \int_0^t \Big{(} (m - \lb \nu_s,\Id\rb) \lb \nu_s, a \vphi' \rb - \lb \nu_s, b \vphi' \rb \Big{)} ds \,,\\[0.8em]
  \ds B^\vphi_t(\nu) & = & \ds  \vphi(0) \int_0^t  \alpha(m - \lb \nu_s,\Id\rb)^2 \, ds\,, \\[0.8em] 
  \ds \tilde B^\vphi_t (\Theta) & = & \ds - \vphi(0) \int_{[0,t)\times \ell^1} \beta q_{0} \Theta(ds\times dq)  \ds  + \vphi'(0)b(0) \int_{[0,t)\times \ell^1} q_{0} \Theta(ds\times dq)\,.
 \end{array}
\end{equation*}
We aim to prove that the limit $(\mu,\Gamma)$ satisfies $ \esp{ \lvert F_t^\vphi(\mu,\Gamma) \rvert } = 0 $. We start by few lemmas.

\begin{lemma} \label{lem:continuity_F}
Let the hypotheses of Theorem \ref{thm:SBD_rescale_limit} hold. Let $\mu^\veps$ be the solution  of (\ref{eq:rescaled_eds}-\ref{eq:rescaled_constraint}) and $\Gamma^\veps$ construct by \eqref{def:Gamma} for each $\veps>0$. Then, for all $\vphi \in \mathcal C_b^1(\Rb_+)$ and $t\geq 0$, up to a subsequence, we have that  $\{F_t^\vphi(\mu^\veps,\Gamma^\veps)\}$  converges to $\{F_t^\vphi(\mu,\Gamma)\}$ in $\Pc\big(\Dc(\Rb_+,w-\Xc)\times (\Yc,w^\#)\big)$. 
\end{lemma}

\begin{proof}
 From their own definition, it appears clearly that  $D^\vphi_t$ and $B^\vphi_t$ are continuous on  $\Dc(\Rb_+,w-\Xc)$, and that $\tilde B^\vphi_t$ is continuous on $(\Yc,w^\#)$. Moreover, the continuity of $t\mapsto \brak{\nu_t}{\vphi}$ entails the continuity of the application $\tilde \nu \mapsto \lb \tilde \nu_t,\vphi\rb$ at $\nu$. Thus, for all $\vphi \in \mathcal C_b^1(\Rb_+)$ and $t\geq 0$, $F_t^\vphi$ is continuous from  $\Dc(\Rb_+,w-\Xc)\times(\Yc,w^\#)$ to $\Rb$ at any point in $\Cc(\Rb_+,w-\Xc)\times \Yc$. Finally, as we said in introduction of Section \ref{sec:proof}, up to a subsequence $\{(\mu^\veps,\Gamma^\veps)\}$ converges to an accumulation point $(\mu,\Gamma)$ in $\Pc\big(\Dc(\Rb_+,w-\Xc)\times(\Yc,w^\#)\big)$ such that $\mu$ belongs to $\Cc(\Rb_+,w-\Xc)$ by Proposition \ref{prop:tight_weak}. It concludes the proof. 
 \end{proof}
Before stating the last lemma which will achieve the proof of Theorem \ref{thm:SBD_rescale_limit}, we introduce a technical result that will be useful to treat the convergence of some terms. 

\begin{lemma}\label{lem_conv_tronque}
 Let $\{\nu^\veps\}$ be a sequence of $\Dc(\Rb_+,w-\Xc)$ such that there exists a function $\Phi_1 \in \Uc_{1,\infty}$ satisfying for any $T>0$
 \[\sup_{\veps>0} \esp{\sup_{t\in[0,T]} \lb \nu_t^\veps , \indic{} + \Phi_1 \rb} < +\infty\,.\] 
 Consider a sequence $\{\vphi^\veps\}$ in $C(\Rb_+)$ such that there exists a constant $K>0$ with $\vphi^\veps(x)\leq K(1+x)$ for all $x$ and $\veps>0$. If $\{\vphi^\veps\}$ converges towards a function $\vphi$ uniformly on the compact sets, then for all $t\in[0,T]$
 \begin{equation*}
 \esp{ \int_0^t \lvert \lb \nu_s^\veps, \vphi^\veps-\vphi \rb \rvert ds }\longrightarrow 0\,,\quad \veps \to 0\,.
 \end{equation*}
\end{lemma}

\begin{proof}
We use the same ideas as in the proof of Lemma \ref{lem:equic_weak}. Let $T>0$, $R>0$, $\veps>0$ and $t \in[0,T]$,  we write
\begin{equation*}
 \lvert \lb \nu_t^\veps, \vphi^\veps-\vphi \rb \rvert \leq  \lvert \lb \nu_t^\veps, (\vphi^\veps-\vphi)\indic{x<R}  \rb \rvert + 2 K \brak{\nu_t^\veps}{(1+x)\indic{x\geq R}}\,.
\end{equation*}
Thus,
\begin{multline*}
\esp{ \int_0^t \lvert \lb \nu_s^\veps, \vphi^\veps-\vphi \rb \rvert \, ds } \leq  T \sup_{x\leq R}\lvert \vphi^\veps(x)-\vphi(x) \rvert \ \esp{ \sup_{t\in[0,T]} \lb \nu_t^\veps, \indic{} \rb } \\ 
 + 2 K T \left( \sup_{x\geq R}\frac{1+x}{\Phi_1(x)}\right) \esp{\sup_{t\in[0,T]}\brak{\nu_t^\veps}{\Phi_1}} \, .
\end{multline*}
We conclude using the moment estimates. Indeed taking the $\limsup$ in $\veps \to 0$ the first term on the right-hand side goes to $0$ and then letting $R\to +\infty$ the second term goes to $0$ too with the property fulfilled by $\Phi_1$.
\end{proof}

\begin{lemma} \label{lem:limitF}
  Under the same assumptions as Theorem \ref{thm:SBD_rescale_limit}, for all $\vphi \in \mathcal C_b^1(\Rb_+)$ and $t\geq 0$
 \begin{equation} \label{eq:limit_F}
  \lim_{\veps\to 0} \esp{ \lvert F_t^\vphi(\mu^\veps,\Gamma^\veps)\rvert} = 0 \,.
 \end{equation}
\end{lemma}

\begin{proof}

First, we remark that by Equation \eqref{eds_weak},
\[  F_t^\vphi(\mu^\veps,\Gamma^\veps) = \Oc^{\veps, \vphi}_t + R^{\veps,\vphi}_t \,,\]
where $R^{\veps,\vphi}_t = \Vc^{\veps,\vphi}_t - D^\vphi_t(\mu^\veps) -  B^\vphi_t(\mu^\veps) - \tilde B^\vphi_t (\Gamma^\veps) = \sum_{i=1}^8 R^{\veps,\vphi, i}_t$ with the terms corresponding to the drift: 

\begin{align*}
 & \ds R^{\veps,\vphi, 1}_t  = (m^\veps-m) \int_0^t   \lb \mu_{s}^\veps , a \vphi' \rb  \, ds \,, \\[0.8em] 
 & \ds  R^{\veps,\vphi, 2}_t  =  \int_0^t  (m^\veps - \lb \mu_s^\veps,\Id\rb) \lb \mu_{s}^\veps , (a^\veps-a) \vphi' \rb  \, ds  \,, \\[0.8em] 
 & \ds R^{\veps,\vphi, 3}_t =  \int_0^t  \lb \mu_{s}^\veps , (b-b^\veps) \vphi' \rb  \, ds  \,, \\[0.8em] 
 & \ds R^{\veps,\vphi, 4}_t =  \int_0^t   (m^\veps - \lb \mu_s^\veps,\Id\rb) \lb \mu_{s}^\veps, ( \Delta_{\veps}(\vphi) - \vphi' ) a^\veps \rb   \, ds  \,, \\[0.8em]
 & \ds R^{\veps,\vphi, 5}_t = \int_0^t  \lb \mu_{s}^\veps, ( \vphi'-\Delta_{-\veps}(\vphi) ) b^\veps \indic{(3\veps,+\infty)} \rb   \, ds \,,
\end{align*}
and to the boundary:
\begin{align*}
&  R^{\veps,\vphi, 6}_t = \int_0^t  \vphi(2\veps) \alpha^\veps (m^\veps - \lb \mu_s^\veps,\Id\rb)(m^\veps - \lb \mu_s^\veps,\Id\rb - \veps^2) -   \vphi(0) \alpha(m - \lb \mu_s^\veps,\Id\rb)^2  \, ds \,,\\[0.8em]
&   R^{\veps,\vphi, 7}_t  =  \int_{[0,t)\times \ell^1}  \vphi(0) \beta q_0  \Gamma^\veps(ds\times dq) - \int_{[0,t)\times \ell^1}  \vphi(2\veps) \beta^\veps q_0  \Gamma^\veps(ds\times dq)\,,\\[0.8em]
&  R^{\veps,\vphi, 8}_t  =  \int_{[0,t)\times \ell^1} \vphi'(2\veps) b^\veps(2\veps) q_{0} \Gamma^\veps (ds\times dq) 
 - \int_{[0,t)\times \ell^1} \vphi'(0) b(0)  q_{0} \Gamma^\veps (ds\times dq)   \,.
\end{align*}
First, using the Burkh\"older inequality, we have
\begin{equation*}
\esp{\Abs{\Oc_t^{\veps,\vphi}}} \leq \Big{(} \esp{\Abs{\Oc_t^{\veps,\vphi}}^2}\Big{)}^{1/2}\leq \Big{(}\esp{\lb \Oc^{\veps,\vphi} \rb_t}\Big{)}^{1/2} \,.
\end{equation*}
Starting from Equation \eqref{eq:martingale_rescale} with $\vphi \in \mathcal C_b^1(\Rb_+)$, we have by the constant $K_0$ in \eqref{eq:k_c_less_Kk} and \eqref{eq:b2_mu,2_bound}
\begin{multline*}
\esp{\lb \Oc^{\veps,\vphi} \rb_t} \leq t \veps 2 K_0 || \vphi||_{\infty}^2
+ t  || \vphi'||_{\infty}^2 \veps^{2}\left(K_a \sup_{s\in[0,t]} u^\veps(s)+K_b\right) \left( \esp{\sup_{s\in[0,t]}\brak{\mu_s^\veps}{\indic{}}}+\esp{\sup_{s\in[0,t]}\brak{\mu_s^\veps}{\mathrm{Id}})} \right)\,.
\end{multline*}
Thus, using the moment estimates in Proposition \ref{prop:moment},
\begin{equation*}
 \lim_{\veps\to 0} \esp{\Abs{\Oc_t^{\veps,\vphi}}} = 0.
\end{equation*}
Then, since $m^\veps := u_{\rm in}^\veps + \lb \mu_{\rm in}^\veps,\Id \rb$ we have that $m^\veps \to m := u_{\rm in} + \lb \mu_{\rm in},\Id \rb$. Indeed $\{u_{\rm in}^\veps\}$ and $\lb \mu_{\rm in}^\veps,\Id \rb$ converge both in law towards a deterministic value, thus they converge in probability and then \as~up to a subsequence. Hence, the expectation of the remainder $\lvert R^{\veps,\vphi, 1}_t \rvert$ goes to $0$ by Proposition  \ref{prop:moment} and hypothesis \eqref{H3}. As well, the expectations of the remainders $\lvert R^{\veps,\vphi, 2}_t \rvert$ to $\lvert R^{\veps,\vphi, 5}_t\rvert$ go to $0$ by the convergence of $a^\veps$ and $b^\veps$ in \eqref{H3} and \eqref{H4}, the moment estimates in Proposition \ref{prop:moment} and the above result in Lemma \ref{lem_conv_tronque}. The remainder $\lvert R^{\veps,\vphi, 6}_t \rvert$ converges to $0$ thanks to \eqref{H1}, the convergence of $m^\veps$ and the estimates in Proposition \ref{prop:moment}.

For the two last remainders $\lvert R^{\veps,\vphi, 7}_t \rvert$ and $\lvert R^{\veps,\vphi, 8}_t \rvert$ , we use a similar strategy as in Lemma \ref{lem_conv_tronque}. Indeed, for any $\vphi\in\Cc_b(\Rb_+)$ and $t\geq 0$,
 \begin{multline*}
  \left\lvert  \int_{[0,t]\times \ell_1^+}  \vphi(0) \beta q_0  \Gamma^\veps(ds\times dq) - \int_{[0,t]\times \ell_1^+}  \vphi(2\veps) \beta^\veps q_0  \Gamma^\veps(ds\times dq)  \right\rvert  \\
   \leq \veps \beta \lvert \vphi(0) - \vphi(2\veps) \rvert  \int_{[0,t]\times \ell_1^+} q_0 \Gamma^\veps(ds\times dq) \\
   + ||\vphi||_{\infty} \lvert \beta -\beta^\veps\rvert \int_{[0,t]\times \ell_1^+} q_0 \Gamma^\veps(ds\times dq)
 \end{multline*}
and we conclude for $\lvert R^{\veps,\vphi, 7}_t \rvert$ remarking that 
\[ \int_{[0,t]\times \ell_1^+} q_0 \Gamma^\veps(ds\times dq) = p_{0,t}^\veps=\mu_t^\veps(\{2\veps\}) \leq \lb \mu_t^\veps,\indic{} \rb \]
and then using \eqref{H2} and Estimation \eqref{eq:mu,1} in Proposition \ref{prop:moment}. The same holds for $\lvert R^{\veps,\vphi, 8}_t \rvert$. This proves that \eqref{eq:limit_F} holds.
\end{proof}

\begin{proof}[Proof of Theorem \ref{thm:SBD_rescale_limit}]
 
 By Propositions \ref{prop:tight_weak} and \ref{prop:tightbord_measure} it follows that along an appropriate subsequence, $\{(\mu^\veps,\Gamma^\veps)\}$  converges as $\veps \to 0$,  to $(\mu,\Gamma)$ in $\Pc\big(\Dc(\Rb_+,w-\Xc)\times (w^\#-\Yc) \big)$ with $\mu\in \Cc(\Rb_+,w-\Xc)$. Moreover, by Lemma \ref{lem:continuity_F}, Lemma \ref{lem:limitF} and \cite[Theorem 3.4]{Billingsley99}, it readily follows that, for any $\vphi\in\Cc_b(\Rb_+)$ and $t\geq 0$,
  \[ \esp{\lvert F^\vphi_t(\mu,\Gamma)\rvert} \leq \liminf_{\veps\to 0} \esp{\lvert F^\vphi_t(\mu^\veps,\Gamma^\veps)\rvert} = 0\,. \]
 Thus, for all $\vphi\in C_b^1(\Rb_+)$ and $t\geq 0$ we have 
 \[F^\vphi_t(\mu,\Gamma) = 0\,,\quad a.s.\]
 Moreover, by Proposition \ref{prop:tight_weak} we have $\{u^\veps\}$ is tight in $\Dc(\Rb_+,\Rb_+)$ and converges (along the same subsequence, up to a modification) to a non-negative $u$ for which it is easy to show that it belongs to $\Cc_c(\Rb_+,\Rb_+)$. By the same arguments as above, we obtain, for all $t\geq 0$, that
 \[ u(t) + \lb\mu_t,\Id\rb = m\,, \ a.s.\]
 %
 
\end{proof}
\begin{proof}[Proof of Theorem \ref{thm:LS_convvague}]
As said before, it is sufficient to consider in \eqref{eq:weak_LS_vague1} the functions $\vphi\in \Cc_c(\Rb_+^*)$, that is separable. Thus, by construction of a set of probability  $0$ as the countable union of probability $0$ sets, \as~for all $t\in\Qb^+$ and $\vphi$ in a dense subset of  $\Cc_c(\Rb_+^*)$ the limit $\mu$ satisfies \eqref{eq:vague_LS}. By continuity in time of $\mu$, we obtain the desired result.
\end{proof}

\subsection{Identification of the occupation measure - Proof of Theorem \ref{thm:LS_convweak}}

Theorem \ref{thm:SBD_rescale_limit} lacks of information because it does not provide any information on $\Gamma$. In this section we aim to identify this measure thanks to a particular limit of the generator $\Hc^\veps$ defined in \eqref{eq:Hveps} and more precisely to its unique stationary measure when it is possible. 

\noindent To that, we focus on $p^\veps$, defined by \eqref{def:p}, through its infinitesimal generator $\Hc^\veps$. As we saw, for each $\veps>0$ the processes $\mu^\veps$ and $p^\veps$ are compactly supported. However the same property is not expected at the limit. Contrary to Proposition \ref{prop:generateur_p}, it requires to make the infinitesimal generator act on functions allowing us to consider sequences in the whole space $\ell_1^+$, not only compactly supported. Therefore, we introduce the domain $\Gc$ defined as
\[ \Gc := \left\{ g:\ell_1\to \Rb \, : \, \exists N\geq 1, \ \exists G\in\Gc_N, \ g(\vphi) = G(\vphi_0,\ldots,\vphi_{N-1}),\ \forall \vphi \in\ell_1   \right\}\,, \] 
where 
\[ \Gc_N := \left\{ G\in\Cc^2(\Rb^N) \, : \, G(0) = 0 \text{ and } \partial_n G \in \Cc^1_c( \Rb^N),\, n=0,\ldots,N-1 \right\} \,.\]
Remark, $\partial_n$ denotes the partial derivatives with respect to the $n^{th}$ variable.

\noindent Now, contrary to Proposition \ref{prop:generateur_p}, using the idea of \cite{Kurtz1992}, we see the infinitesimal generator $\Hc^\veps$ as an operator coupling the action of $p^\veps$ and $\mu^\veps$. In order to do that, we define, thanks to Assumption \ref{hyp:coef_BD}, for all $g$ in the domain $\Gc$, and for all $(\nu,q)\in\Xc\times\ell_1^+$,   the operator

\begin{multline*}
\widetilde \Hc^{\veps} g (\nu,q)= \ds \frac{g(q+ \veps \indic{0}) - g(q)}{\veps} \alpha^\veps c(c-\veps^2) \vphantom{\sum_{i\geq k}}
 + \ds \frac{g(q- \veps \indic{0}) - g( q)}{\veps} \beta^\veps q_0 \vphantom{\sum_{i\geq k}} \\
  + \veps^{-(1-r_a)}   \sum_{n\geq0}  \frac{g(q+ \veps (\indic{n+1}-\indic{n})) - g( q)}{\veps}  \frac{a^{\veps}(\veps(n+2))}{\veps^{r_a}}   c q_n\\
 + \veps^{-(1-r_b)}  \sum_{n \geq 1} \frac{g(q- \veps (\indic{n}-\indic{n-1})) - g( q)}{\veps} \frac{b^{\veps}(\veps(n+2))}{\veps^{r_b}} q_n \vphantom{\sum_{i\geq k}}\,,
\end{multline*}
where $c$ is now replaced by $m^{\veps}-\lb \nu ,\Id \rb$ contrary to the previous definition of $\Hc^\veps$ in \eqref{eq:Hveps}. The operator $\widetilde \Hc^{\veps}$ is well-defined on the whole domain $\Gc$ since: for all $g$ in $\Gc$, there exists a $N'\geq 0$ such that  
$$g(q+ \veps (\indic{n+1}-\indic{n}))= g(q- \veps (\indic{n}-\indic{n-1})) =g(q)\,,$$
for all $q \in \ell_1^+$ and $n\geq N'$. It readily follows from its definition that, for all $g \in \Gc$,
 \begin{equation}\label{eq_on_H}
 g( p^\veps_t) = g(p^\veps_0)+  \int_0^t [\widetilde \Hc^\veps g ](\mu_s^\veps,p^\veps_s)\,ds + \Oc_t^{\veps, g}\,,
 \end{equation}
 where $\Oc_t^{\veps\, g}$ is a martingale. Remark that, taking $r:= \min(r_a,r_b)<1$ and multiplying this generator by $\veps^{(1-r)}$, at the limit some terms will vanish depending on the value of $r_a$ and $r_b$. The latter depend on the behavior of $a$ and $b$ around $0$. Indeed, a direct consequence of Assumption \ref{hyp:coef_BD} implies that for all $n\in\Nb$,

 \begin{equation}\label{def:an,bn}
  \dfrac{a^{\veps}(\veps(n+2))}{\veps^{r_a}} \underset{\veps\to 0}{\rightarrow}a_n := \overline{a}\,  (n+2)^{r_a}, \quad \dfrac{b^{\veps}(\veps(n+2))}{\veps^{r_b}} \underset{\veps\to 0}{\rightarrow} b_n:=\overline{b} \, (n+2)^{r_b}\,.
 \end{equation}
with $\{a_n\}$ and $\{b_n\}$ positive by Assumption \ref{hyp:coef_BD}. We are in position to define the limit operator: for all $g$ in $\Gc$ and  $(\nu,q)\in\Xc\times\ell_1^+$,
 
\begin{equation}\label{gene-H}
[\widetilde \Hc\ g](\nu,q) := \sum_{n\geq 0} Dg[q](1_n) (J_{n-1}(\nu,q) -J_n(\nu,q))
= \sum_{n\geq 0} Dg[q](1_{n+1}-1_n) J_{n}(\nu,q)\,,
\end{equation}
where $Dg$ is the Fr\'echet derivative of $g$, by convention $J_{-1}=0$ and for all $n\geq 0$ \begin{equation}\label{def:BD_fluxes}
 J_n(\nu,q) := \begin{cases}
\ a_n c q_n - b_{n+1}q_{n+1}\,, & \text{if } r_a=r_b< 1\,,\\
\ a_n c q_n\,,  & \text{if } r_a<r_b\,, \text{ and } r_a< 1\,,\\
\ -b_{n+1}q_{n+1}\,, & \text{if } r_b<r_a\,, \text{ and } r_b< 1\,.
\end{cases}
\end{equation}
with $c=m-\brak{\nu}{\mathrm{Id}}$, the constant $m$ arising from the limit in Theorem \ref{thm:SBD_rescale_limit}. Note the similarity with the classical Becker-D\"oring fluxes for the deterministic equations. The next theorem identifies the limit $\Gamma$ as a stationnary measure.

\begin{theorem} \label{prop:cv_bord}
In addition to the hypotheses of Theorem \ref{thm:SBD_rescale_limit}, suppose that Assumption \ref{hyp:coef_BD} holds. Then, it exists $(\gamma_t)_{t\geq0}$ a $\Pc(v-\ell_1^+)$-valued optional process such that $\Gamma = \gamma_t(dq) dt$ and with probability one, \ae~$t\geq0$ and all $g\in\Gc$
\[\int_{\ell^+_1} [\widetilde \Hc \, g](\mu_t,q) \gamma_t (dq) = 0\,.\]
Moreover, let $\rho$ be defined by \eqref{def:rho}. On a time interval $[t_0,t_1]$ such that the limit $c_t>\rho$ for all $t_0\leq t \leq t_1$,  then \as~$\gamma_t=\delta_{\mathbf{0}}$ the Dirac measure at the null sequence of $\ell_1^+$.
\end{theorem}

\begin{remark}
This theorem tells us that the limit $\Gamma$ is the product of a probability measure on $\ell_1^+$ and the Lebesgue measure on $\Rb_+$. The probability measure is a stationnary measure associated to the limit generator $\widetilde \Hc$. So we are able to completely identify $\Gamma$ only when we can ensure the operator $\widetilde \Hc$ has a unique stationary measure. This operator is connected to a constant-particle Becker-D\"oring system. If we investigate the stationary solutions of the generator through its dynamics, there are two cases: either the time-dependent solution trends to an equilibrium or the solution escapes to infinity (larger and larger clusters are formed), see for instance \cite{Ball1986,Ball1988,Carr1999}. Surprisingly, we cannot identify the stationary measure in the first case since the equilibrium is parametrized by the total number of clusters, which is unknown here (it is \textit{not} $\lb \mu_t,\indic{}\rb$). It provides an infinity of stationary solutions and one can show  (see \ref{ssec:det_BD}) that the support of the stationary measure belongs to the set of all possible stationary solutions and not only one. 
On the other hand, when there is no equilibrium, the solution vaguely converges to the unique zero-solution which provides an identification between the stationary measure and the long time solution of the Becker-D\"oring system. In this case, $p^\veps_n$, for a fix $n$, which is a very small cluster, goes to $0$ when $\veps$ goes to $0$. In contrary, larger and larger (in $n$) clusters are formed, which induces at the limit clusters of size $x>0$. This is the case when we have an identifiable boundary condition.
\end{remark}
 
Let us introduce some lemmas before the proof of Theorem \ref{prop:cv_bord}. First, note the Fr\'echet derivative of a function in $\Gc$ can be expressed in a simpliest way. By definition, for any $g\in\Gc$, there exists an integer $N$ and a function $G$ in $\Cc^2( \Rb^N) $ such that $g(\vphi) = G(\vphi_0,\ldots,\vphi_{N-1})$, for any $\vphi\in\ell_1$. Hence, the Fr\'echet derivative $Dg$ is, for all $q$ in $\ell_1^+$ and $\vphi$ in $\ell_1$,

\begin{equation} \label{eq:frechet}
 Dg[q](\vphi)=\sum_{n=0}^N \partial_n \, G (q_1, \dots, q_N) \ \vphi_n\,, 
\end{equation}
and so $Dg[q](1_n)=\indic{0\leq n\leq N}\ \partial_n \, G (q_1, \dots, q_N)$. This shows that the generator $\widetilde \Hc$ is well-defined on $\Gc$ since the sum is actually finite.

We now state a lemma on the convergence of the generators $\veps^{(1-r)} \widetilde \Hc^{\veps}$ to $\widetilde \Hc$ along the processes $p^\veps$ and $\mu^\veps$.%

\begin{lemma}\label{lem:reste_gene_bord}
 Under the same assumptions as Theorem \ref{prop:cv_bord}, we have, for all $T>0$ and $g$ in $\Gc$,
 $$\lim_{\veps \rightarrow 0}\Eb\left[\sup_{t\in\left[0,T\right]} \Bigg\vert\int_0^{t}\veps^{(1-r)}[\widetilde \Hc^{\veps}g ](\mu_s^{\veps},p^{\veps}_s)- [\widetilde \Hc g ](\mu_s^{\veps},p^{\veps}_s)\,ds\Bigg\vert\right]=0\,.$$
\end{lemma}

\begin{proof} We start with the case $r=r_a=r_b$ so that the fluxes $J_n$ in \eqref{def:BD_fluxes}  are in the more general form. Let us fix $T>0$ and $g$ in $\Gc$. Remark first that, thanks to \eqref{eq:frechet} and by Taylor's theorem, there exists a positive constant $K_g$ such that for all $q$ in $\ell_1^+$ and $\vphi$ in $\ell_1$, we have the following bounds
 \begin{align*}
 \quad \vert Dg[q](\vphi) \vert \leq K_g\|\vphi\|_{\ell_1}, \quad \Big\vert \ds \frac{g(q+ \veps \vphi) - g(q)}{\veps}- &\,Dg[q](\vphi) \Big\vert
 \leq K_g\, \veps\|\vphi\|^2_{\ell_1}\,,
 \end{align*}
and, therefore
\[ \Big\vert \ds \frac{g(q+ \veps \vphi) - g(q)}{\veps} \Big\vert
 \leq K_g \|\vphi\|_{\ell_1}(1+\veps \|\vphi\|_{\ell_1})\,.\]
%
%
%
From the definition of $\widetilde \Hc^{\veps}$ and $\widetilde \Hc$ it readily follows that for all $s\in[0,T]$

\begin{equation*}
 \left\vert \veps^{(1-r)} [\widetilde \Hc^{\veps}\, g](\mu_s^{\veps},p^{\veps}_s)-[\widetilde \Hc\, g](\mu_s^{\veps},p^{\veps}_s) \right\vert \leq I_{1, s}^\veps + \sum_{n =0}^N  I_2^\veps(n)p_{n,s}^{\veps} + \sum_{n = 1}^N  I_3^\veps(n)p_{n,s}^{\veps} \,,
\end{equation*}
with
\begin{equation*}
 \begin{array}{l}
  \ds I_{1, s}^\veps =  \veps^{(1-r)}K_g (1+\veps)(\vert \alpha^\veps u^{\veps}(s)^2\vert+ \vert \beta^\veps p_{0,s}^{\veps}\vert ) \,,\\[0.8em]
  \ds I_2^\veps(n) =   \Bigg\vert\frac{g(p_{n,s}^{\veps}+ \veps (\indic{n+1}-\indic{n})) - g( p_{n,s}^{\veps})}{\veps}  \frac{a^{\veps}(\veps(n+2))}{\veps^{r_a}}   u^{\veps}(s) 
  \ds -Dg[p_{n,s}^{\veps}](1_{n+1}-1_n) a_n u^\veps(s)  \Bigg\vert \,,\\[0.8em]
  \ds I_3^\veps(n) =  \Bigg\vert\frac{g(p_{n,s}^{\veps}- \veps (\indic{n}-\indic{n-1})) - g( p_{n,s}^{\veps})}{\veps} \frac{b^{\veps}(\veps(n+2))}{\veps^{r_b}}  -Dg[p_{n,s}^{\veps}](1_{n-1}-1_n) b_n\Bigg\vert \,.
 \end{array}
\end{equation*}
%
%
Using the constant $K_0$ derived in \eqref{eq:k_c_less_Kk} and \eqref{eq:b2_mu,2_bound} we get

\[  \esp{\sup_{s\in[0,T]} I_{1, s}^\veps} \leq \veps^{(1-r)}K_g 2K_0\,. \]
Then, for all $n$ in $\Nb$, we have $I_2^\veps(n) \leq {I_2^\veps}'(n) + {I_2^\veps}''(n)$ with
\begin{align*}
&{I_2^\veps}'(n) =  \Bigg\vert\frac{g(p_{n,s}^{\veps}+ \veps (\indic{n+1}-\indic{n})) - g( p_{n,s}^{\veps})}{\veps}  -Dg[p_{n,s}^{\veps}](1_{n+1}-1_n)  \Bigg\vert\quad \frac{a^{\veps}(\veps(n+2))}{\veps^{r_a}} u^{\veps}(s)  \, ,\\
& {I_2^\veps}''(n) = \hphantom{\Bigg\vert}  \big\vert Dg[p_{n,s}^{\veps}](1_{n+1}-1_n) \big \vert\,  \Bigg\vert\frac{a^{\veps}(\veps(n+2))}{\veps^{r_a}}-a_n \Bigg\vert  u^{\veps}(s) \, .
 \end{align*}
Remark that, by \eqref{def:an,bn}, for each $n$ it exists $K_{a,n}$ such that $a^{\veps}(\veps(n+2))/\veps^{r_a}\leq K_{a,n}$. Thus, denoting by $K_1$ the bound on $u^\veps$ in \eqref{eq:C_uni_bound}, we then end up with the bound
\begin{equation*}
 I_2^\veps(n) \leq 2 K_g K_1\Bigg(\veps K_{a,n} + \Bigg\vert\frac{a^{\veps}(\veps^{\beta}(n+2))}{\veps^{r_a\beta}}-a_n \Bigg\vert \Bigg)\,.
\end{equation*}
In a similar way, we obtain for $n \geq 1$
\begin{equation*}
I_3^\veps(n) \leq 2 K_g \Bigg( \veps K_{b,n} +  \Bigg\vert\frac{b^{\veps}(\veps^{\beta}(n+2))}{\veps^{r_b\beta}} -b_n \Big\vert\Bigg)\,.
\end{equation*}
Finally, from the above estimates and using that $p^\veps_{n,s} \leq \lb \mu_s^\veps, \indic{} \rb$ (by definition), 
%
%
the convergence of the $a^{\veps}$'s and $b^{\veps}$'s in \eqref{def:an,bn}, and by the moment estimates in Proposition \ref{prop:moment}, it concludes the proof for $r_a=r_b$. In the other cases, the proof is similar.
\end{proof}
Before proving Theorem \ref{prop:cv_bord} and more precisely in order to use the convergence of $(\mu^{\veps}, \Gamma ^{\veps})$ towards $(\mu, \Gamma)$, a last lemma of continuity is necessary.

\begin{lemma}\label{lemma_continuity_Hg}
For all $g$ in $\Gc$ and all $t\geq 0$, the function \[(\nu, \Theta) \mapsto \int_{[0,t]\times\ell^+_1} [\tilde \Hc \, g](\nu_s,q) \Theta(ds\times dq)\] is continuous at any point of $\Cc(\Rb_+, w-\Xc)\times \Yc$.
\end{lemma}

\begin{proof}
Let us fix $g$ in $\Gc$, $t \geq 0$ and a point $ (\nu, \Theta) $ of $\Cc(\Rb_+, w-\Xc)\times \Yc$. For any sequence $ (\nu^{\veps}, \Theta^{\veps}) $ converging to  $ (\nu, \Theta) $ when $\veps$ goes to $0$, we have
\begin{multline}\label{continuity_Hg}
\Bigg\vert  \int_{[0,t]\times\ell^+_1} [\widetilde \Hc \, g](\nu_s^{\veps},q) \Theta^{\veps}(ds\times dq) -  \int_{[0,t]\times\ell^+_1} [\widetilde \Hc \, g](\nu_s,q) \Theta(ds\times dq)\Bigg\vert\hfill \\
 \leq  \Bigg\vert  \int_{[0,t]\times\ell^+_1} [\widetilde \Hc \, g](\nu_s,q) \Theta^{\veps}(ds\times dq) -  \int_{[0,t]\times\ell^+_1} [\widetilde \Hc \, g](\nu_s,q) \Theta(ds\times dq)\Bigg\vert \\
 \hfill + \int_{[0,t]\times\ell^+_1}\Big\vert  [\widetilde \Hc \, g](\nu_s^{\veps},q)-[\widetilde \Hc \, g](\nu_s,q) \Big\vert\Theta^{\veps}(ds\times dq) \, .
\end{multline}
The convergence of the first term on the right-hand side of \eqref{continuity_Hg} to $0$ is due to the convergence of $\Theta^{\veps}$ to $\Theta$ in $\Yc$ for the $weak^\#$ topology  (see Appendix) since we have the bound

\[\vert[\widetilde \Hc \, g](\nu_s,q) \vert \leq 2 K_g \Big( (\sup_{0\leq n\leq N} a_n) (m + \sup_{s \in [0,t]} \lb \nu_s, \Id \rb )  +\sup_{1\leq n\leq N+1}  b_n\Big) \sum_{n=0}^{N+1} q_n\,,\]
where the  constant $K_g$ is the same as in the previous proof.

Consider now the second term. For all $n\in \Nb$, $s\in[0,t]$ and $q$ in $\ell_1^+$, the following bound on the flux $J_n$ is obtained
\[ \vert J_n(\nu_s^{\veps},q) - J_n(\nu_s,q)\vert \leq a_n \vert \lb \nu_s^{\veps}, \Id \rb -\lb \nu_s, \Id \rb \vert\, q_n\,.\]
Therefore, we get
\[ \Big\vert  [\widetilde \Hc \, g](\nu_s^{\veps},q) -  [\widetilde \Hc \, g](\nu_s,q)\Big\vert \leq 2 K_g \Big(\sup_{0\leq n\leq N} a_n \Big) \Big(\sup_{s \in [0,t]} \vert \lb \nu_s^{\veps}, \Id \rb -\lb \nu_s, \Id \rb \vert\Big)\sum_{n=0}^{N} q_n\,.\]
This inequality in particular shows the continuity of the map $ \tilde \nu \mapsto  [\tilde \Hc \, g](\tilde \nu,q)$ for all $g$ in $\Gc$ and $q$ in $\ell_1^+$ and gives the convergence to $0$ of the second term on the right-hand side in \eqref{continuity_Hg} when $\veps$ goes to $0$. The result is proved. \end{proof}
We are now in position to identify the limit $\Gamma$ and prove Theorem \ref{prop:cv_bord}. 

\begin{proof}[Proof of Theorem \ref{prop:cv_bord}]
We may rewrite \eqref{eq_on_H} as 
\begin{align*}
  \veps^{(1-r)} \Oc^{\veps,g}_t & =\veps^{(1-r)} ( g( p^\veps_t)  - g(p^\veps_0)) - \int_0^t [\widetilde \Hc\, g](\mu^\veps_s,p^\veps_s) - e^\veps_t \,,\\
  & \phantom{:}=\veps^{(1-r)} ( g( p^\veps_t)  - g(p^\veps_0))  - \int_{[0,t]\times\ell^+_1} [\widetilde \Hc\, g] (\mu^\veps_s,q) \Gamma^\veps(ds\times dq) - e^\veps_t\,,
\end{align*}
 with
 \[e^\veps_t = \int_0^{t}\veps^{(1-r)}[\widetilde\Hc^{\veps} u](\mu_s^{\veps},p^{\veps}_s)- [\widetilde \Hc\, g](\mu_s^{\veps},p^{\veps}_s)\,ds \,.\]
 Thus, for all $T>0$,  we have by Lemma \ref{lem:reste_gene_bord} that $\esp{\sup_{t\in[0,T]} \lvert e^\veps_t\rvert} \to 0$ when $\veps \to 0$. We may check that the limit %
 \[\int_{[0,t]\times\ell^+_1} [\widetilde \Hc \, g](\mu_s,q) \Gamma(ds\times dq)\,, \]
obtained by Lemma \ref{lemma_continuity_Hg}, is a martingale, which is continuous and of bounded variations and hence must be constant, in fact equal to $0$. 
 Thus, for each $g\in \Gc$ and $t\geq 0$ 
 \[\int_{[0,t]\times\ell^+_1}[\widetilde \Hc \, g](\mu_s,q) \Gamma(ds\times dq) =0\,, \quad a.s.\]
 Using \cite[Lemma 1.4]{Kurtz1992} with a slight adaptation along the proof, it exists $(\gamma_t)_{t\geq0}$ a $\Pc(\ell_1^+)$-valued optional process, such that  for all $t\geq 0$ and  $B\in \Bc(\ell_1^+)$ 
 \[ \Gamma([0,t]\times B ) = \int_{[0,t]} \gamma_s(B) ds\,, \quad a.s. \]
 and since the functions $q \in \ell_1^+ \mapsto q_i$ are $\Gamma$-integrable it readily follows that for all $t\geq 0$ and  $g\in\Gc$
  \[\int_{[0,t]\times\ell^+_1} [\widetilde \Hc \, g](\mu_s,q) \gamma_s(dq) ds = 0,\quad a.s. \]
 Hence, by separability of $\Gc$ (as $\Cc_c^1(\Rb^n)$ is separable), with probability one, we have 
  \[\int_{\ell^+_1}[\widetilde \Hc \, g](\mu_t,q) \gamma_t(dq) =0,\quad a.e \ t \geq 0  \text{ and } \forall g\in\Gc\,. \]
Then, thanks to Proposition \ref{prop:mean_stationary_BD} in Appendix, for a fixed $\nu\in\Xc$ such that $c=m-\langle \nu, \Id\rangle>\rho$, the operator $[\tilde \Hc \cdot](\nu,\cdot)$ has a unique stationary distribution $\pi_{\nu}=\delta_{\mathbf{0}}$ in $\Pc(v-\ell_1^+)$, the Dirac measure at the null sequence in $\ell_1^+$, \ie~  satisfying
  \[\int_{\ell^+_1}[\widetilde \Hc g](\nu,q) \pi_{\nu}(dq)  = 0\,, \quad \forall g \in \Gc \,.\]
 Therefore, on a time interval $[ t_0, t_1 ] $ such that $u(s)= m-\langle \mu_s, \Id\rangle>\rho$, we can conclude that the process $(\gamma_s)_{s \in [t_0,t_1]}$ is deterministic and equals to $\delta_{\mathbf{0}}$. This proves the result.
 \end{proof}

\begin{proof}[Proof of Theorem \ref{thm:LS_convweak}]
The proof readily follows from Theorem \ref{prop:cv_bord} combining to Theorem \ref{thm:SBD_rescale_limit}. 
\end{proof}

\section{Discussion} \label{sec:disc}

The link between the discrete size Becker-D\"oring  model and the continuous size Lifshitz-Slyozov model has already been studied within the context of deterministic model by \cite{Collet2002,Laurencot2002a}. We used a similar approach to those previous studies, in the sense that we have introduced a scaling parameter, linked to the initial number of particles, and investigated the limit when this scaling parameter tends to zero. The main difference is that in both studies \cite{Collet2002,Laurencot2002a}, the authors obtained convergence results in a vague topology, that is, the topology of convergence against compactly supported test functions. The authors in \cite{Laurencot2002a} were able to extend the convergence in the space of integrable densities against $xdx$, which do not see the boundary as well (as the weight $x$ vanishes at the boundary $x=0$). Thus, these results were restricted in practice to cases where the characteristics exit the domain (for which well-posedness do not require specification of the boundary term).

Concerning the regularity imposed on the rate coefficients, we essentially have the same hypotheses as \cite{Laurencot2002a}. However, our choice of scalings slightly differs. In both studies \cite{Collet2002,Laurencot2002a}, only the nucleation rate is slowing down compared to other aggregation rates; this condition being crucial to prevent explosion, mainly to get the crucial moment estimates in Proposition \ref{prop:moment}. In our case, we also need to slow down the de-nucleation rate, which allows us to obtain time control in Proposition \ref{prop:tight_weak}. Such a part has only been  proved for compactly supported functions in the previous works \cite{Collet2002,Laurencot2002a}, without the extra scaling of the de-nucleation rate. We interpret physically our extra-condition as being asymptotically an irreversible nucleation hypothesis. We conjecture that the reversible nucleation case can still be managed by a similar strategy (with more involved compactness estimates), and will yield a different boundary condition. 

Let us now illustrate our results and point future research directions with the help of numerical simulations. The discrete stochastic Becker-D\"oring model (Definition \ref{def:rescaleSBD}) can be simulated using stochastic simulation algorithms (SSA, or Gillespie algorithm). Thus, the stochastic trajectories can be compared to numerical solution of the deterministic limit problem obtained in our main Theorem \ref{thm:LS_convweak}. 

{\bf 1.} We illustrate in Figure \ref{fig1} the perfect match between the stochastic and the deterministic numerical solutions, as long as $\veps$ is small enough.

{\bf 2.} The stochastic trajectories may be used to study relevant first passage times. For instance, starting with only $M=\frac{1}{\veps^2}$ free particles, we can compute the time needed to reach a given large size of order $\frac{1}{\veps}$. Indeed, using our scaling, in Figure \ref{fig2} we see that the mean time to reach this size is of order $\frac{1}{\veps}$, as $\veps \to 0$. 

{\bf 3.} We show in Figure \ref{fig3} that the stochastic model may deviate strongly from the limit deterministic model, for small but positive $\veps$, although it does not contradict our limit theorem. Indeed, using coefficient rates such that the characteristics exit the domain ($a(0)u(t)-b(0)<0$), the deterministic model predicts that the pure free particle initial condition is an equilibrium (if $u(0)=m$, then $u(t)\equiv m$ for all $m$). However, for $\veps>0$, after a stochastic (asymptotically very long) time, the stochastic model switch to a different phase, $u^\veps(t)$ being close to $0$ in the stochastic trajectories. We expect second order approximations and large deviation theory to explain such a behavior.


This work may have several applications. For instance, the rare protein assembly in neuro-degenerative diseases is being intensively modeled by aggregation-fragmentation models \cite{Prigent2012}. The understanding of the behavior of such models is thus a first necessary step towards deciphering the mechanism of the diseases (inside an organism or in \textit{in vitro} polymerization experiments). In such a context, deriving a discrete (and stochastic) aggregation-fragmentation model is appealing for its simplicity and is more intuitive. However, it has several drawbacks compared to continuous models. First, the time consuming stochastic simulations were the main limitation of the numerical exploration of the behavior of the stochastic Becker-D\"oring model in \cite{Yvinec2012} (see also in a deterministic context, \cite{Prigent2012}). Hence, deriving an approximate continuous limiting model is important for fast accurate numerical simulations of discrete Becker-D\"oring model, as standard, well-known and fast numerical schemes are widely available for continuous size-structured PDE models \cite{Banks2014}. 
\pagebreak
\newpage

\begin{figure}[!h]
\includegraphics{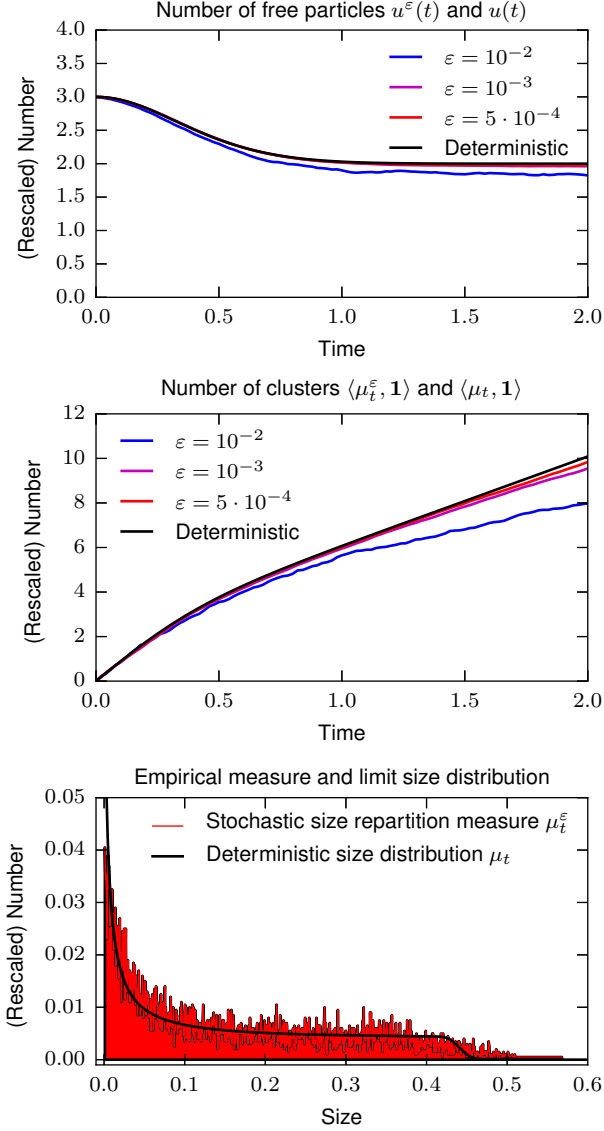} 
\caption{\textbf{Agreement between numerical simulations and our limit theorem.} We plot the time evolution of the (rescaled) number of free particles $u_t^{\veps}$ (top) and the total (rescaled) number of clusters $\brak{\mu_t^\veps}{\indic{}}$ (middle) for different $\veps$ (see legend), together with the deterministic solution of the moment equations obtained from the weak form of the LS equation \eqref{eq:weak_LS} (in black). Down, we plot one snaphshot at time $t=1$ of the measure-valued solution $\mu_t^\veps$ for $\veps=10^{-3}$, together with the numerical solution of the LS equation (standard upwind scheme). We used the scaling given in Section \ref{sec:rescale} with constant rate coefficients $a^\veps(x)\equiv 1$ and $b^\veps(x)\equiv 2$, $\alpha^\veps=1$, and $b^\veps=1$, under initially incoming characteristics, \it{i.e} $u(0)=m=3>\rho=\frac{b}{a}=2$.}
\label{fig1}
\end{figure}
\pagebreak

\begin{figure}[!h]
\includegraphics{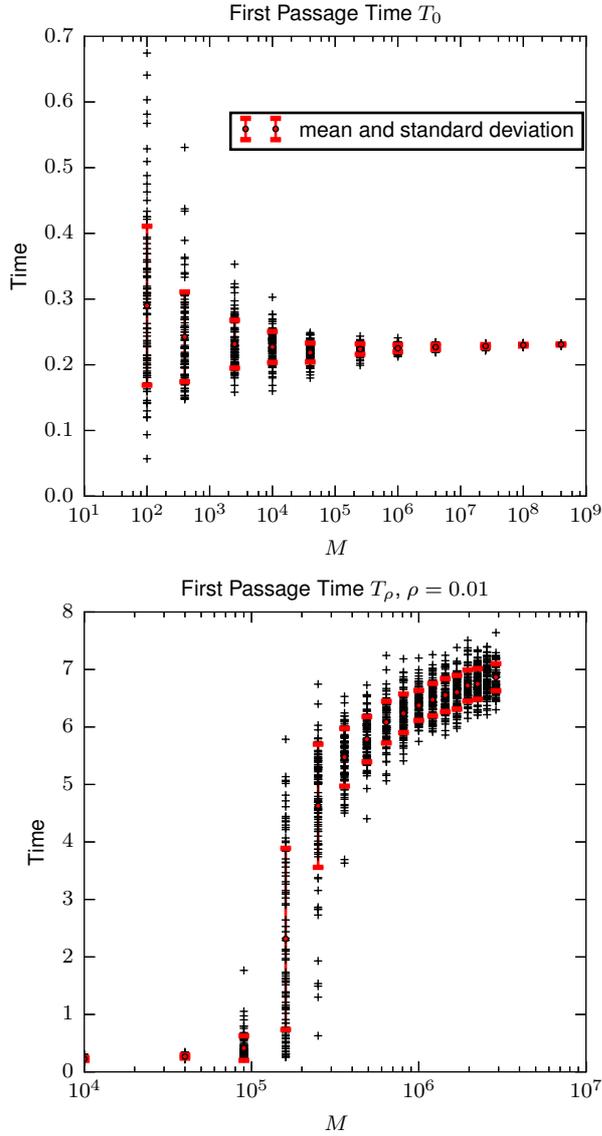} 
\caption{\textbf{The time scale of formation of large clusters follows the scaling of our limit theorem.} We run stochastic simulations of the original discrete Becker-D\"oring system given by the generator \eqref{eq:gene_BD_stoch}, for various initial mass $M=1/\veps^2$. We define the stopping time $T_0=\inf\{t\geq 0, C_{\lfloor 1/\veps \rfloor}(t)=1 \mid C_1(0)=1/\veps^2 \}$ and $T_{\rho}=\inf\{t\geq0, C_{\lfloor 1/\veps \rfloor}(t)=\lfloor\rho/\veps\rfloor \mid C_1(0)=1/\veps^2 \}$. The time $\veps T_1$ is reported in the first figure, and $\veps T_\rho$, for $\rho=0.01$, in the second, and both are plotted as a function of the initial mass $M$. We use kinetic coefficients  as follows: $a^\veps(x)=5\veps^2$, $b^\veps(x)=x$, $\alpha^\veps=\veps^4$, $\beta^\veps=\veps^2$. The fact that the two quantities converge as $M\to\infty$ is consistent with our results.}
\label{fig2}
\end{figure}
\pagebreak

\begin{figure}[!h]
\includegraphics{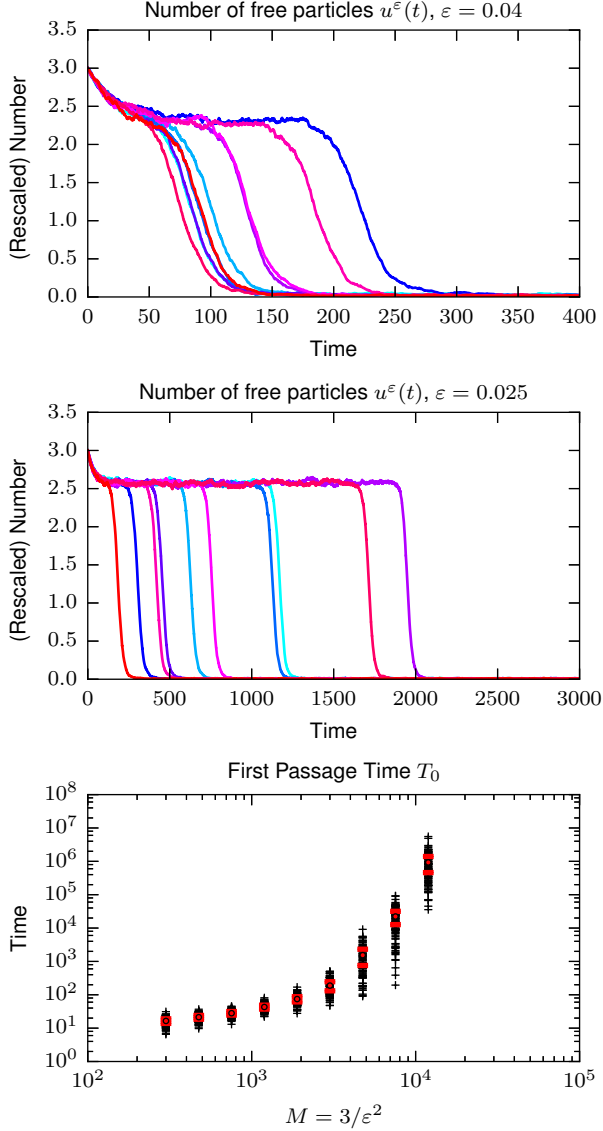} 
\caption{\textbf{Deviations from the deterministic limit.} We perform the same numerical simulations as in Figure \ref{fig1}, but using for rate coefficients $a^\veps(x)=x$ and $b^\veps(x)\equiv 1$, $\alpha^\veps=\beta^\veps=1$ and $m^\veps=3$. With these rate coefficients, note that the characteristics of the LS equation \eqref{eq:weak_LS} are always outgoing ($a(0)=0$ and $b(0)>0$). We start with a pure free particle initial condition, corresponding to $u^\veps(0)=m^\veps=3$. We plot the time evolution of the number of free particles $u_t^\veps$ for ten independent trajectories with $\veps=4 . 10^{-2}$ (top) and $\veps=2,5 .10^{-2}$ (middle). We observe that the numerical solutions largely differ from one to each other, mostly by the time at which the number of free particles drastically goes down. This time corresponds to the time a cluster of size greater than the critical size $X_c=1/u(t)$ has been formed (the size for which $a(x)u(t)-b(x)$ becomes positive). In contrast, the deterministic solution of the LS equation \eqref{eq:weak_LS} predicts a constant level of free particle $u(t)$. Down, we plot the realization of the stopping time $T_0=\inf\{t\geq 0, C_{\lfloor 1/\veps \rfloor}(t)=1 \mid C_1(0)=1/\veps^2 \}$, as in Figure \ref{fig2}. The latter seems to grow exponentially, as expected in large deviation theory.}
\label{fig3}
\end{figure}

\pagebreak

\newpage

\appendix

 \section{Topology and Compactness}
 
 In the sequel $E$ is a Polish space (a separable topological space which is completely metrizable) and we consider its underlying Borel $\sigma$-algebra $\Bc(E)$.

 \subsection{The space $\Xc$} \label{app:space_X}

 Let $(h_i)_{i\geq 1}$ be a countable sequence (possibly finite) of nonnegative real-valued measurable functions on $E$. We define
 
 \[ \Xc(E) : = \left\{ \vphantom{\sum_{i=1}^n} \nu \in \Mc_b(E)\, : \,  \lb \nu ,   h_i    \rb < +\infty\,, \ \forall i\geq1  \right\}  \,  \]
equipped with the weak topology, denoted by $(\Xc,w)$ or alternatively $w-\Xc$, \ie~the  coarsest topology that makes continuous $\nu \mapsto \lb \nu, \vphi\rb$ and $\nu \mapsto \lb \nu,  h_i \vphi\rb$ for all $\vphi\in\Cc_b(E)$ and $i\geq 1$. Remark, for all $i\geq 1$ and $\nu\in \Xc(E)$, we can define the density measure $h_i \cdot \nu \in\Mc_b(E)$ by $\lb h_i\cdot \nu ,\vphi \rb =  \lb  \nu ,h_i \vphi \rb$  for any $\vphi\in\Cc_b(E)$, see \cite[Chap. IX \S2.2]{bourbaki_int_chap9}.
  
 \begin{lemma} \label{lem:X_polish}
  The space $(\Xc(E),w)$ is a Polish space. Let $\rho_\Xc$ be defined, for any $(\nu,\mu)\in\Xc(E)\times\Xc(E)$, by
  \begin{equation*}
  \rho_\Xc(\nu,\mu) :=  \sum_{i\geq 0} 2^{-(i+1)}(  1\wedge \rho( h_i \cdot  \nu,  h_i  \cdot \mu)) \,,
  \end{equation*}
  where $\rho$ is the Prohorov metric on $\Mc_b(E)$ and the convention $h_0=\indic{}$.  Then, $\rho_\Xc$ is a complete metric equivalent to the weak topology on $\Xc(E)$.
 \end{lemma}
 See for instance \cite[Section 6]{Billingsley99} for a definition of the Prohorov metric. 
 \begin{proof}
  The properties of $(\Xc(E),w)$ derive from its identification to the space $\{\nu \in \Mc_b \, : \,  h_i\cdot \nu \in \Mc_b(E)\,, \ \forall i\geq 0\}$.
 \end{proof}

\begin{lemma}[A criterion of weakly relatively compactness in $\Xc$]\label{lem:compactX}
Let $\Kc$ be a subset of  $\Xc$. Suppose it exists a non-negative measurable function $H$ on $E$ such that, for all $n>0$ the sets $K_n = H^{-1}([0,n])$ are compact in $E$. Moreover suppose it exists $n_0\geq0$ such that for all $i\geq 0$, and $x\in  K_{n_0}^c$, we have  $h_i(x)  \leq H(x)$. Assume further that it exists $\Phi \in \Uc_\infty$ (defined in Section \ref{sec:technical}) such that
\[ \sup_{\nu\in \Kc} \lb \nu ,  \indic{} +  H  + \Phi( H )  \rb < +\infty\,.\]
Then, $\Kc$ is relatively compact in $(\Xc(E),w)$. Moreover, let $\{\nu^\veps\}$ be a sequence in $\Kc$ and assume that $h_i$ is continuous for some $i\geq 1$, then up to a subsequence it exists $\nu\in\Xc(E)$ such that $\nu^\veps\to\nu$  in the weak topology as $\veps\to 0$ and  
\[ \lb \nu^\veps ,h_i  \rb \to \lb \nu ,h_i  \rb \,, \quad \text{ as }\ \veps\to 0\,.\]
\end{lemma}

\begin{proof}
 The aim is to link these bounds to a criterion of weakly relatively compactness in $\Mc_b(E)$. Let $\nu$ in $\Kc$, then for $n\geq n_0$ and for all $i\geq 1$
 \begin{align*} 
 \lb (\indic{} + h_i ) \cdot \nu ,  \indic{K_n^c} \rb = \int_{K_n^c} \frac{1}{ H}  H +  \frac{ h_i }{\Phi(H )}\Phi( H)  \nu(dx)
 \leq \left( \frac 1 n +\sup_{y\geq n}  \frac{y}{\Phi(y)} \right) \sup_{\nu \in \Kc} \lb \nu ,   H + \Phi( H) \rb\,.
 \end{align*}
 When $n\to +\infty$ the right hand side goes to $0$. It yields $(\indic{} +  h_i ) \cdot \nu$ belongs to a weakly relatively compact set of $\Mc_b(E)$, see \cite[Chap. IX \S5.5, Theorem 2]{bourbaki_int_chap9}.  Let $\{\nu^\veps\}$ be a sequence in $\Kc$, there exist $\mu_i\in \Mc_b(E)$ and a subsequence (still indexed by $\veps$) such that $\{(\indic{}+ h_i  )\cdot \nu^\veps\}$ weakly converges to $\mu_i$ in $\Mc_b(E)$. Hence, by a diagonal process, for all $i\geq 1$ and for any $\vphi \in \Cc_b(E)$,
 \[ \lb (\indic{}+  h_i  )\cdot \nu^\veps ,\vphi \rb \to \lb \mu_i , \vphi  \rb \, .\]
 Since $\vphi =(\indic{}+ h_i )^{-1}$ is a continuous bounded function, it yields in particular, for all $i\geq 0$, $\nu^\veps \to \nu_i := (\indic{}+ h_i )^{-1} \cdot \mu_i$ in $\Mc_b(E)$. By the uniqueness of the limit, we define $\nu = \nu_i$. It readily follows that $\nu\in \Xc(E)$ and $\nu^\veps\to \nu$ in $w-\Xc(E)$. The last remark comes from the fact we can take  $ \vphi = h_i / (1+h_i)$ which is a bounded and continuous function.
  \end{proof}
Let us details two classical examples about the control of $x$-moments and a more complex applications useful for our purpose.  
  
  \begin{example}Take $E= [0,+\infty)$, the functions $h_1 = H = \Id$. It readily follows a compact criterion for the measure space   $\Xc([0,+\infty))$ defined by $\left\{ \vphantom{\sum_{i=1}^n} \nu \in \Mc_b([0,+\infty))\, : \,  \lb \nu ,  \Id   \rb < +\infty \right\}$.    
  \end{example}
  
 \begin{example} Take $E= [0,+\infty)$, the functions $h_i : x\mapsto x^i$ for $i = 1,\ldots, p$ and $H=h_p$. We have for all $x>1$ that $h_i(x) \leq H(x)$.  It readily follows a compact criterion for  $\Xc(E) : = \left\{ \vphantom{\sum_{i=1}^n} \nu \in \Mc_b(E)\, : \,  \lb \nu ,  h_i   \rb < +\infty \,, \ i = 1,\ldots, p \right\}$. Remark that $\Phi$ can be chosen as $x \mapsto x^{p+1}$.  
 \end{example}
  
  \begin{example} Take $E= \ell_1^+$ with the vague topology (see Remark \ref{topo_l1}), the functions $h_i(q) = q_i$ and $H(q) = \|q\|_{\ell_1}$. The $h_i$ are continuous and the pre-image by $H$ of any bounded set is compact in $E$ for the vague topology. It readily follows from the previous lemma a compact criterion for $\Xc(E) : = \left\{ \vphantom{\sum_{i=1}^n} \nu \in \Mc_b(E)\, : \,  \lb \nu ,  h_i   \rb < +\infty \,, \ \forall i\geq 0\right\}$. Remark that, if $\nu^\veps \to \nu$ in $w-\Xc$ then 
  \[\int_{\ell_1^+} q_i \nu^\veps(dq) \to \int_{\ell_1^+} q_i \nu(dq)\,,\ \forall i\geq 0\,.\]
  But since the norm is not continuous for the vague topology it appears clearly that we cannot hope the convergence of $\int_{\ell_1^+} \|q\| \nu^\veps(dq)$ to $\int_{\ell_1^+} \|q\| \nu(dq)$.
  \end{example}

\subsection{The space $\Yc$} \label{app:space_Y}

We proceed here to a slight adaptation of \cite{Kurtz1992}. Let $(h_i)_{i\geq1}$ be a sequence of measurable function on a Polish space $E$, and for any $t\geq 0$ we consider $\Xc([0,t]\times E)$ defined similarly to the previous section as a subset of $\Mc_b([0,t]\times E)$. Now, we consider the space
\[ \Zc(\Rb_+\times E) : = \Big\{ \Theta \in \Mc(\Rb_+\times E)\ : \ \forall t \geq 0, \, \Theta^t \in \Xc([0,t]\times E) \Big\}\,,\]
where $\Theta^t$ denotes the restriction of $\Theta$ to $[0,t]\times E$. We endow this space with the metric $\rho_\Zc$ given, for any $\Theta$ and $\Gamma$ belonging to $\Zc(\Rb_+\times E)$, by
\[ \rho_\Zc (\Theta,\Gamma)=\int_0^\infty e^{-t}\, 1 \wedge  \rho^t_\Xc( \Theta^t,\Gamma^t)\,dt\,,\]
 where $\rho_\Xc^t$ is the modified Prohorov metric on $\Xc([0,t]\times E)$. This metric defines the weak$^\#$ topology on $\Zc(\Rb_+\times E)$ and the space is denoted by $(\Zc(\Rb_+\times E),w^\#)$. Note that  a sequence $\{\Theta^\veps\} \subset \Zc$ converges in  $\rho_\Zc$ if and only if $\{\Theta^{\veps\,, t} \}$ converges in $\rho^t_\Xc$ for almost every $t$. The next three lemmas follow \cite[Appendix 2.6]{Daley2002} and \cite{Kurtz1992}.

\begin{lemma}
 The space $\Zc(\Rb_+\times E)$ equipped with the weak$^\#$ topology is a Polish space. 
\end{lemma}

\begin{lemma}
 The subspace of $\Zc(\Rb_+\times E)$ given by
 \[\Yc(\Rb_+\times E) := \{ \Theta \in \Zc(\Rb_+\times E)  \, : \, \Theta([0,t]\times E) = t \}\]
 and equipped with the topology induced by $\rho_\Zc$ is a Polish space.
\end{lemma}

\begin{proof}
 We just remark that $\Yc(\Rb_+\times E)$ is a closed suset of $\Zc(\Rb_+\times E)$. Indeed, if $\{\Theta^\veps\}$ converges to $\Theta$ in $\rho_\Zc$, then  $\{\Theta^{\veps\,, t}\} \to \Theta^t$ in $\rho_t$ if and only if $\Theta^{\veps}([0,t]\times E) \to \Theta([0,t]\times E)$ as $\veps\to 0$.
 \end{proof}

\begin{lemma}[A criterion of weakly relatively compactness in $\Yc$] \label{lem:compactY}
 Let $\{t_k\}$ be a non-decreasing sequence in $\Rb_+$ such that $\lim_{k\to+\infty} t_k = +\infty$. Then, the set
 
 \[\Big\{ \Theta \in \Yc(\Rb_+\times E) \, : \, \forall k, \ \exists \text{ weak  compact } \Kc_k \subset \Xc([0,t]\times E),\ \Theta^{t_k} \in \Kc_k\Big\}\]
 is a compact of $(\Yc,\rho_\Zc)$.
\end{lemma}

\section{Stationary states and measures for Becker-D\"oring}\label{ssec:det_BD}
The aim of this appendix is to investigate the stationary measures of a modified Becker-D\"oring model represented for a given $\nu\in\Xc$ by an operator $ [\widetilde \Hc \cdot] (\nu, \cdot)$ defined by \eqref{gene-H}.

\noindent A stationary measure of such a model is a probability measure $\pi$ on $\ell_1^+$ solution of 
\begin{equation}\label{mean_equiv_stationary}
 \int_{\ell^+_1}\ [\widetilde \Hc g] (\nu, q) \pi(dq)= 0\,, \quad  \forall g\in \Gc\,.
\end{equation}
We start by studying the stationary states of $ [\widetilde \Hc \cdot] (\nu, \cdot)$, that is, the sequences $\overline q \in \ell_1^+$ satisfying
\begin{multline} \label{equiv_stationary}
 [\widetilde \Hc g] (\nu, \overline q)= 0\,, \quad \forall g\in \Gc
 \Leftrightarrow  \sum_{n\geq 0} Dg[\overline q](1_n) (J_{n-1}(\nu,\overline q) -J_n(\nu, \overline q)) = 0\,, \quad \forall g\in \Gc\,,
\end{multline}
where $J_{-1}=0$ and the Becker-D\"oring fluxes $J_n$ for $n\geq 0$ are given by \eqref{def:BD_fluxes} and are recalled in the next proposition.

\begin{proposition} \label{prop:stationary_BD}
Let $\nu\in \Xc$ such that $c=m-\lb \nu ,\Id \rb \geq 0$, the exponents $r_a, r_b$, the coefficients $\overline a$, $\overline b$ be given by Assumption \ref{hyp:coef_BD}, $\rho$ be defined by \eqref{def:rho} and the sequences $(a_n)_{n \in \Nb}$ and $(b_n)_{n\geq 1}$ by \eqref{def:an,bn} .
\begin{enumerate}
 \item In the case $r_a<r_b$, $r_a <1$, the Becker-D\"oring fluxes $J_n(\nu,q) = a_n c q_n$ for all $n\in \Nb$. If $c>\rho= 0$, then the unique stationary state is
\[ \overline q_n = 0\,, \quad \forall n\geq 0\,.\]
\item In the case $r_b<r_a$, $r_b <1$, the Becker-D\"oring fluxes $J_n(\nu,q) = - b_{n+1} q_{n+1}$ for all $n\in \Nb$. Then the stationary states are all given by
\[ \overline q_n = 0\,, \ \forall n\geq 1\,, \text{ and }\ \overline q_0  \geq 0\,.\] In particular, $\overline q_0 = \lVert \overline q \rVert_{\ell_1}$.
\medskip
\item In the case $r_a=r_b<1$, the Becker-D\"oring fluxes $J_n(\nu,q) = a_n c q_n - b_{n+1} q_{n+1}$ for all $n\in \Nb$. Denoting $Q_n= (\prod_{i=0}^{n-1} a_i/b_{i+1} )$ for all $n\in\Nb^*$ and $Q_0 = 1$, we have
$$1/\rho= \limsup_{n\to + \infty} Q_n^{1/n}= \dfrac{\overline a}{\overline b }\,.$$
Moreover, if $0\leq c < \rho$, then the stationary states are all given by
\[ \overline q_n = (Q_n c^n) \overline q_0\,, \ \forall n\geq 1\,, \text{ and }\ \overline q_0  \geq 0\,.\]
In particular $\lVert \overline q \rVert_{\ell_1} =  (\sum_{n\geq 0} Q_n c^n) \overline q_0$.
\item In the case $r_a=r_b<1$, if $c > \rho$, then the unique stationary state is
\[ \overline q_n = 0\,, \quad \forall n\geq 0\,.\]
\end{enumerate}

\end{proposition}

\begin{proof}
Note that, by \eqref{equiv_stationary}, the state $\overline q$ is stationary if for all $g\in\Gc$,
\[ \sum_{n\geq 0} Dg[\overline q](1_n) (J_{n-1}(\nu,\overline q) -J_n(\nu, \overline q)) = 0\,.\]
In particular, applying for some functions $g$ depending on only one term of sequences of $\ell_1^+$, that is, $ g(q)= G (q_n)$ with $G \in \Cc^2_c$ for instance for a fixed $n$, we obtain
\[   J_n(\nu,\overline q) = 0\,, \quad \forall n\,.\]
The points 1. and 2. directly follows.
%
%

\noindent In the cases 3. and 4. in which $r=r_a=r_b$, we also deduce by induction that the stationary states are of the form
\[\overline q_n = (Q_n c^n) \overline q_0\,, \quad \forall n\geq 1\,.\]
Let us then prove that the radius of convergence of the sum $\sum Q_n c^n$ is $\rho={\overline b}/{\overline a }$. By the Cauchy-Hadamard Rule, this radius is $1/ \limsup_{n\to + \infty} Q_n^{1/n}$. Note that, since the $a_n$'s and the $b_n$'s are given by \eqref{def:an,bn}, the term $Q_n$ can be written for $n\geq 1$ as
\begin{align*}
Q_n= \prod_{i=0}^{n-1} \dfrac {a_i}{b_{i+1}} = \left(\dfrac{\overline a}{\overline b}\right)^n \ \dfrac{(2)^{r}\cdots ( n-1+2)^{r}}{(1+2)^{r}\cdots ( n+2)^{r}}= \left(\dfrac{\overline a}{\overline b}\right)^n \ \dfrac{(2)^{r}}{ (n+2)^{r}}\,.
\end{align*}
We thus immediately conclude that $\limsup_{n\to + \infty} Q_n^{1/n}= {\overline a}/{\overline b}= 1/\rho$ and the result is proved.

\noindent And, as $\overline q$ has to belongs to $\ell_1^+$, if $0\leq c <\rho$ the sum is convergent and we obtain point 3. If now $c>\rho$, the sum is not convergent so the unique solution is the null-sequence, giving point 4.

\end{proof}
We can now proceed to the identification of the stationary measures of a modified Becker-D\"oring model but, unfortunately, only in the cases 1. and 4. of the previous proposition.

\begin{proposition}\label{prop:mean_stationary_BD}
Let $\nu\in \Xc$ such that $c=m-\lb \nu ,\Id \rb > 0$. In the cases 1. and 4. of Proposition \ref{prop:stationary_BD}, the unique stationary measure of the modified Becker-D\"oring model represented by the operator $ [\widetilde \Hc \cdot] (\nu, \cdot)$ is the Dirac measure $\delta_{\mathbf{0}}$ at the null-sequence.
\end{proposition}

\begin{remark}
In the cases 2. and 3. of Proposition \ref{prop:stationary_BD}, there is no uniqueness of the stationary states vanishing all the fluxes but an infinite collection parametrized by the first component $\overline q_0$. Because of this, there is also no uniqueness of the stationary measures of the associate modified Becker-D\"oring model. Indeed, following the proof of Proposition \ref{prop:mean_stationary_BD} here below, we can only conclude that, in these cases, a stationary measure is supported on all the stationary states. For instance, any probability measure, convex combination of Dirac measures at stationary states, is a stationary measure. This particularly implies that we are not able to identify the limit of the occupation measures $\Gamma$ in Theorem \ref{prop:cv_bord} in the cases 2. or 3.
\end{remark}
Before proving this result, we state a useful lemma requiring the introduction of a new space of functions from $\ell_1^+$ to $\Rb$. We denote by $\overline \Gc$ the set of functions $f$ from $\ell_1^+$ to $\Rb$ such that there exist $N' \geq 0$ and a function $F$ in $\Cc_c^1(\Rb^{N'})$ satisfying $f(q)= F(q_0,\dots, q_{N'-1})$ for all $q$ in $\ell_1^+$. This space can be understood as the set of functions obtained by taking the Fr\'echet Derivative of a function $g$ in $\Gc$ applied to a canonical sequence $\indic{n}$ for a given $n$, that is, $Dg[q](1_n)$.

\begin{lemma}\label{stationary_measure}
Let $V$ be a continuous function from $v-\ell_1^+$ to $\Rb$ such that there exist $N\geq 0$ and a continuous function $\overline V$ from $\Rb^N$ to $\Rb$ with $V(q)= \overline V(q_0,\dots, q_{N-1})$ for all $q$ in $\ell_1^+$. A probability measure $\pi$ satisfying
\begin{equation}\label{def_pi}
\int_{\ell_1^+} f(q) V(q)\, \pi(dq) = 0\,, \quad \forall f \in \overline \Gc\,,
\end{equation}
is supported on $Z(V):=\lbrace q \ \vert \  V(q)=0 \rbrace $.
 \end{lemma}
\begin{proof}
First note that all the measures supported on $Z(V)$ satisfy \eqref{def_pi}. Conversely let us prove that a measure $\pi$ such that \eqref{def_pi} holds is supported on $Z(V)$. We introduce $\Omega= \supp \ \pi \cap Z(V)^c$ with $Z(V)^c= \ell_1^+\setminus Z(V)$. We recall that the space $\ell_1^+$ is endowed with the vague topology and is metrizable as $(\Mc_b(\Nb),v)$.

\noindent We start by assuming that the interior of $\Omega$ is nonempty, ie
$
 \mathring{\Omega} \neq \emptyset\,, 
$
and let us fix an element $q^1$ in $ \mathring{\Omega}$. By definition $V(q^1)$ is either positive or negative. We here suppose that $V(q^1)>0$ (the other case is similar). Since the function $V$ is continuous, there exists $r_1>0$ such that $V$ is positive on $\overline \Bb(q^1,r_1)\subset  \mathring{\Omega}$, the closed ball of radius $r_1$ and center $q_1$. We then consider a function $f$ in $\overline \Gc$ such that
$$\begin{cases}
f(q)>0 \quad \text {for all $q$ in the open ball $ \Bb(q^1,r_1/2)$}\,,\\
f(q)\geq 0 \quad \text {for all $q$ in $\overline \Bb(q^1,r_1)$}\,,\\
f(q)=0 \quad \text{otherwise}\,.
\end{cases}$$
Applying \eqref{def_pi} with $f$, we have
\begin{equation*}
 0=\int_{\ell_1^+}f(q)V(q)\pi(dq)= \int_{\overline \Bb(q^1,r_1)}f(q)V(q)\pi(dq)\,.
\end{equation*}
Since $f(q)V(q)\geq 0$ for all $q\in \overline \Bb(q^1,r_1)$ and $f(q)V(q)> 0$ on $ \Bb(q^1,r_1/2)$, there is a contradiction. Thus, the set $\Omega$ has an empty interior and is therefore discrete. The measure $\pi$ restricted to $\Omega$ can be written as 
$$\pi_{\vert \Omega}= \sum_{i \in I} \lambda_i \delta_{q^i}\,,$$
with the $q^i$'s in $\Omega$ and $\lambda_i\geq 0$. Now using a test function $f$ in $\overline \Gc$ such that $f(q^i)= V(q^i)$ and $f(q^j)=0$ for all $j \neq i$, we can deduce that $\lambda_i=0$. This proves the result.
 \end{proof}
We are now in position to prove Proposition \ref{prop:mean_stationary_BD}.
\begin{proof}[Proof of Proposition \ref{prop:mean_stationary_BD}]
Assume that $\pi$ is a stationary measure, that is, satisfying \eqref{mean_equiv_stationary}.
For all $i\in\mathbb N$ and $f\in \overline \Gc$, we consider the function $g^i$ in $\Gc$ such that for all $q$ in $\ell_1^+$
$$Dg^i[q](1_n) = f(q) \indic{n=i}\,,$$
that is,
$$g^i(q)=\int_0^1f(tq)q_i \,dt\,.$$
Applying \eqref{mean_equiv_stationary} with $g=g^i$, we get
\[\int_{\ell_1^+} f(q) (J_{i-1}(\nu,q) -J_i(\nu,q)) \pi(dq) = 0\,.\]
Thus the measure $\pi$ satisfies
\[\int_{\ell_1^+} f(q) (J_{i-1}(\nu,q) -J_i(\nu,q)) \pi(dq) = 0\,, \quad \forall f \in \overline\Gc \text{ and } i \in \Nb\,.\]
Since for $i=0$, we have
\[\int_{\ell_1^+} f(q) J_0(\nu,q) \pi(dq) = 0\,, \quad \forall f \in \overline\Gc\,,\]
we can deduce by induction that, for all $n$ in $\Nb$ and $f$ in $ \overline\Gc$,
\[\int_{\ell_1^+} f(q) J_n(\nu,q) \pi(dq) = 0 \,.\]
Finally, applying Lemma \ref{stationary_measure} with $V=J_n(\nu,\cdot )$ for all $n$ in $\Nb$, the measure $\pi$ is supported on the sequences of $\ell_1^+$ vanishing all the fluxes $J_n$. By Proposition \ref{prop:stationary_BD}, the result follows.
\end{proof}

\end{document}